\documentclass[a4paper,reqno,11pt,oneside]{amsart}\usepackage[left=3cm,right=3cm,top=3.5cm,bottom=3.5cm]{geometry}
\usepackage[colorlinks=true,urlcolor=blue,
citecolor=red,linkcolor=blue,linktocpage,pdfpagelabels,
bookmarksnumbered,bookmarksopen]{hyperref}
\usepackage[english]{babel}
\usepackage{graphicx}
\usepackage[toc,page,header]{appendix}
\usepackage{mathrsfs}
\usepackage{amssymb,amsmath, amsfonts, amsthm, mathtools}
\usepackage{xcolor}
\usepackage[shortlabels]{enumitem}
\usepackage{lineno}
\makeatletter
\@namedef{subjclassname@2020}{%
  \textup{2020} Mathematics Subject Classification}
\makeatother

\DeclareMathOperator\divergence{div}

\DeclareMathOperator\jac{Jac}

\newcommand{\R}{\mathbb{R}}

\usepackage{latexsym}

\allowdisplaybreaks[1]

\numberwithin{equation}{section}

\theoremstyle{plain}
\newtheorem{theorem}{Theorem}[section]
\theoremstyle{plain}
\newtheorem{prop}[theorem]{Proposition}
\theoremstyle{plain}
\newtheorem{lemma}[theorem]{Lemma}
\theoremstyle{plain}
\newtheorem{cor}[theorem]{Corollary}
\theoremstyle{definition}

\theoremstyle{definition}
\newtheorem{definition}[theorem]{Definition}
\theoremstyle{definition}
\newtheorem{remark}[theorem]{Remark}
\theoremstyle{definition}

\theoremstyle{plain}

\begin{document}

\addtolength\jot{3pt}

\renewcommand{\labelenumi}{\textit{(\roman{enumi})}}

\makeatletter
\def\author@andify{%
  \nxandlist {\unskip ,\penalty-1 \space\ignorespaces}%
    {\unskip {} \@@and~}%
    {\unskip , \penalty-2}%
}
\makeatother

\title[Energy stability for a class of semilinear elliptic problems]{Energy stability for a class of semilinear elliptic problems}

%\author{ D\MakeLowercase{anilo} G\MakeLowercase{regorin} Afonso}

\thanks{Research partially supported by Gruppo Nazionale per l'Analisi Matematica, la Pro\-ba\-bi\-li\-t\`a e le loro Applicazioni (GNAMPA) of the Istituto Nazionale di Alta Matematica (INdAM)}

\author{Danilo Gregorin Afonso}

\address[Danilo Gregorin Afonso]{Dipartimento di Matematica Guido Castelnuovo, Sapienza Universit\`a di Roma, Italy} \email{danilo.gregorinafonso@uniroma1.it}

\author{Alessandro Iacopetti}
\address[Alessandro Iacopetti]{Dipartimento di Matematica ``G. Peano", Universit\`a di Torino, Via Carlo Alberto 10, 10123 Torino, Italy}
\email{alessandro.iacopetti@unito.it}
%\address[Alessandro Iacopetti]{Dipartimento di Matematica ``G. Peano", Universit\`a di Torino, Via Carlo Alberto 10, 10123 Torino, Italy}
%\email{alessandro.iacopetti@unito.it}

\author{Filomena Pacella}
\address[Filomena Pacella]{Dipartimento di Matematica Guido Castelnuovo, Sapienza Universit\`a di Roma, Piazzale Aldo Moro 5, 00185 Roma, Italy}
\email{pacella@mat.uniroma1.it}

%\subjclass[2020]{35B99, 35J25, 35J61, 35N25, 49Q10, 49J45}
%
%\keywords{Overdetermined elliptic problems, semilinear elliptic problems, shape optimization in unbounded domains}
%
%\thanks{\emph{Acknowledgements.} Research partially supported by Gruppo Nazionale per l'Analisi Matematica, la Pro\-ba\-bi\-li\-t\`a e le loro Applicazioni (GNAMPA) of the Istituto Nazionale di Alta Matematica (INdAM)}


\maketitle

%\begin{abstract}
%In this paper we consider semilinear elliptic problems on bounded domains $\Omega\subset \mathcal{C}$, with homogeneous mixed boundary value conditions (relative to $\mathcal{C}$), where  $\mathcal{C}$ is a given unbounded Lipschitz domain of $\R^N$. We show that for a nondegenerate solution $u_\Omega$ it is possible to define a domain energy functional $\tilde\Omega\mapsto T(\tilde\Omega)$, for all $\tilde\Omega$ in a suitable neighborhood of $\Omega$ in $\mathcal{C}$. Then we introduce the notion of energy-stationary couple $(\Omega, u_\Omega)$ under a volume constraint and study the stability of the energy $T(\Omega)$ with respect to small volume-preserving perturbations. If $\mathcal{C}$ is a cone or a half-cylinder we determine sufficient conditions for stability/instability depending on the geometry of $\mathcal{C}$.
%\end{abstract}

\begin{abstract}
    In this paper, we consider semilinear elliptic problems in a bounded domain $\Omega$ contained in a given unbounded Lipschitz domain $\mathcal C \subset \mathbb R^N$. Our aim is to study how the energy of a solution behaves with respect to volume-preserving variations of the domain $\Omega$ inside $\mathcal C$. Once a rigorous variational approach to this question is set, we focus on the cases when $\mathcal C$ is a cone or a cylinder and we consider spherical sectors and radial solutions or bounded cylinders and special one-dimensional solutions, respectively. In these cases, we show both stability and instability results, which have connections with related overdetermined problems.
\end{abstract}

\section{Introduction}

\label{sec:introduction}

Let $\mathcal C \subset \mathbb R^N$, $N \geq 2$, be an unbounded uniformly Lipschitz domain and let $\Omega \subset \mathcal C$ be a bounded Lipschitz domain with smooth relative boundary $\Gamma_\Omega \coloneqq \partial \Omega \cap \mathcal C$. More precisely, we assume that $\Gamma_\Omega$ is a smooth manifold of dimension $N-1$ with smooth boundary $\partial\Gamma_\Omega$. We set $\Gamma_{1, \Omega} \coloneqq \partial \Omega \setminus \overline \Gamma_\Omega$ and assume that $\mathcal H^{N - 1}(\Gamma_{1, \Omega}) > 0$, where $\mathcal H^{N - 1}$ denotes the $(N - 1)$-dimensional Hausdorff measure. Hence $\partial \Omega = \Gamma_\Omega \cup \Gamma_{1, \Omega} \cup \partial \Gamma_\Omega$.

We consider the following semilinear elliptic problem:
\begin{equation}
\label{eq:pde}
\left\{
\begin{array}{rcll}
- \Delta u & = & f(u) & \quad \text{ in } \Omega \\[4pt]
u & = & 0 & \quad \text{ on } \Gamma_\Omega \\[2pt]
\displaystyle \frac{\partial u}{\partial \nu} & = & 0 & \quad \text{ on } \Gamma_{1, \Omega}
\end{array}
\right.
\end{equation}
where $f: \mathbb R \to \mathbb R$ is a locally $C^{1,\alpha}$ nonlinearity and $\nu$ denotes the exterior unit normal vector to $\partial \Omega$.

Let $u_\Omega$ be a positive weak solution of \eqref{eq:pde} in the Sobolev space $H_0^1(\Omega \cup \Gamma_{1, \Omega})$, which is the space of functions in $H^1(\Omega)$ whose trace vanishes on $\Gamma_\Omega$. By standard variational methods we have that under suitable hypotheses on $f$ such a solution exists and is a critical point of the energy functional
\begin{equation}
\label{eq:energy_functional}
J(v) = \frac{1}{2} \int_\Omega |\nabla v|^2 \ dx - \int_\Omega F(v) \ dx, \quad v \in H_0^1(\Omega \cup \Gamma_{1, \Omega}),
\end{equation}
where $F(s) = \int_0^s f(\tau) \ d\tau$.

A classical example of  a nonlinearity for which a positive solution exists for any domain $\Omega$ in $\mathcal C$ is the Lane-Emden nonlinearity, namely 
\begin{equation}\label{eq:LaneEmden}
f(u) = u^p,\ \text{with} \begin{cases} 1 < p < \frac{N + 2}{N - 2} & \text{if $N \geq 3$},\\[2pt] 1 < p < + \infty &\text{if $N = 2$.}\end{cases}
\end{equation}
 In this case, $u_\Omega$ can be obtained, for instance, by minimizing the functional $J$ on the Nehari manifold
\begin{equation*}
\label{eq:Nehari_manifold}
\mathcal N(\Omega) = \{v \in H_0^1(\Omega \cup \Gamma_{1, \Omega}) \setminus\{0\} \ : \  J'(v)[v] = 0\} .
\end{equation*}

Given the unbounded region $\mathcal C$, an interesting question is to understand how the energy $J(u_\Omega)$ behaves with respect to variations of a domain $\Omega$ inside $\mathcal C$. In particular, one could ask whether the energy $J(u_\Omega)$ increases or decreases by deforming $\Omega$ into a domain $\widetilde \Omega$ sufficiently close to $\Omega$ and with the same measure. 

Loosely speaking, one could consider the function $\Omega\mapsto T(\Omega) = J(u_\Omega)$ and study it in a suitable ``neighborhood" of $\Omega$. Under this aspect, domains $\Omega$ which are local minima of $T$ could be particularly interesting. This question could be attacked by differentiating $T(\Omega)$ with respect to variations of $\Omega$ which leave the volume invariant and studying the stability or instability of its critical points. However, since \eqref{eq:pde} is a nonlinear problem and solutions of \eqref{eq:pde} are not unique in general, it is not clear a priori how to well define the functional $T(\Omega)$.

We will show in Section \ref{sec:general_unbounded_domains} that for nondegenerate solutions $u_\Omega$ of \eqref{eq:pde} the energy functional $T(\Omega)$ is well defined for domains obtained by small deformations of $\Omega$ induced by vector fields which leave $\mathcal C$ invariant. %and which preserve the volume.

We remark that the study of the stationary domains of the energy functional $T(\Omega)$ with a volume constraint is strictly related to the overdetermined problem obtained from \eqref{eq:pde} by adding the condition that the normal derivative $\frac{\partial u}{\partial \nu}$ is constant on $\Gamma_\Omega$, see Proposition \ref{prop:energy_stationary_char}. This is well-known for a Dirichlet problem in $\mathbb R^N$ and when $T(\Omega)$ is globally defined for all domains $\Omega \subset \mathbb R^N$ (as in the case of the torsion problem, i.e. $f \equiv 1$). It has been observed in \cite{PacellaTralli2021} and \cite{IacopettiPacellaWeth2022} in the relative setting of the cone.

The existence or not of domains that are local minimizers of the energy and their shapes obviously depend on the unbounded region $\mathcal C$ where the domains $\Omega$ are contained. In this paper, we consider unbounded cones and cylinders, in which there are some particular domains that, for symmetry or other geometric reasons, could be natural candidates for being local minimizers of the energy.\\

Let us first describe the case when $\mathcal C$ is a cone $\Sigma_D$ defined as
\begin{equation}
\Sigma_D \coloneqq \{x \in \mathbb R^N \ : \ x = tq, \ q \in D, \ t > 0\},
\end{equation}
where $D$ is a smooth domain on the unit sphere ${\mathbb{S}^{N - 1}}$.

In $\Sigma_D$ we consider the spherical sector $\Omega_D$ obtained by intersecting the cone with the unit ball centered at the origin, i.e. $\Omega_D=\Sigma_D\cap B_1$. In $\Omega_D$ we can consider a radially symmetric solution $u_D$ of problem \eqref{eq:pde}, for the nonlinearities $f$ for which they exist. Obviously, $u_D$ is a radial solution of the analogous Dirichlet problem in the unit ball $B_1$.

In Section \ref{sec:cone} we show that, whenever $u_D$ is a nondegenerate solution of \eqref{eq:pde}, then the pair $(\Omega_D, u_D)$ is energy-stationary in the sense of Definition \ref{def:energy_stationary} and investigate its ``stability" as a critical point of the energy functional $T$, which is well defined for small perturbations of $\Omega_D$ (see Sections \ref{sec:general_unbounded_domains} and \ref{sec:cone}).

The main result we get is that the stability of $(\Omega_D, u_D)$ depends on the first nontrivial Neumann eigenvalue $\lambda_1(D)$ of the Laplace-Beltrami operator $- \Delta_{\mathbb{S}^{N - 1}}$ on the domain $D \subset \mathbb{S}^{N - 1}$ which spans the cone. In particular, we obtain a precise threshold for stability/instability which is independent of the nonlinearity, and on the radial positive solution considered, whenever multiple radial positive solutions exist. Let us remark that for several nonlinearities the radial positive solution is unique (see \cite{NiNussbaum1985}). For example, this is the case if $f(u) = u^p$, $p > 1$.

To state precisely our result we need to introduce the first eigenvalue $\widehat \nu_1$ of the following singular eigenvalue problem:
\begin{equation}
\label{eq:singular_problem_intro}
\begin{cases}
- z'' - \frac{N - 1}{r}z' - f'(u_D)z = \frac{\widehat \nu}{r^2} z \quad \text{in } (0, 1) \\
z(1) = 0
\end{cases}
\end{equation}
This problem arises naturally when studying the spectrum of the linearized operator $- \Delta - f'(u_D)$. We refer to Section \ref{sec:cone} for more details.

\begin{theorem}
\label{thm:stability_cone}
Let $\Sigma_D$ be the cone spanned by the smooth domain $D \subset \mathbb{S}^{N - 1}$, $N \geq 3$, and let $\lambda_1(D)$ be the first nonzero Neumann eigenvalue of the Laplace-Beltrami operator $- \Delta_{\mathbb{S}^{N - 1}}$ on $D$. Let $u_D$ be a radial positive solution of \eqref{eq:pde} in the spherical sector $\Omega_D$. We have:
\begin{enumerate}
\item if $- \widehat \nu_1 < \lambda_1(D) < N - 1$, then the pair $(\Omega_D, u_D)$ is an unstable energy-stationary pair;

\item if $\lambda_1(D) > N - 1$, then $(\Omega_D, u_D)$ is a stable energy-stationary pair.
\end{enumerate}
\end{theorem}

\begin{remark}
    \label{rem:Remark_on_thm_stability_cone}
    The case $N = 2$ is special and in this case, the overdetermined torsion problem has been completely solved in \cite{PacellaTralli2020} using that the boundary of any cone in dimension $2$ is flat. In the nonlinear case, the condition $N \geq 3$ arises from the study of an auxiliary singular problem (see Proposition \ref{prop:h(0)_equals_0}). It is important to observe that the singular eigenvalue $\widehat \nu_1$ which appears in \textit{(i)} is larger than $-(N - 1)$ for all autonomous nonlinearities $f(u)$ (see \cite[Proposition 3.4]{CiraoloPacellaPolvara2023}). Thus the condition $\lambda_1(D) \in (- \widehat \nu_1, N - 1)$ is consistent. 
\end{remark}

Let us comment on the meaning of Theorem \ref{thm:stability_cone}. The statement \textit{(ii)} will be proved by showing that the quadratic form corresponding to the second derivative of the energy functional, with a fixed volume constraint, is positive definite in all directions. This means that the spherical sector locally minimizes the energy among small volume preserving perturbations of $\Omega_D$ and of the corresponding radial solution $u_D$. 

On the contrary, when $- \widehat \nu_1 < \lambda_1(D) < N - 1$, by \textit{(i)} we have that the pair $(\Omega_D, u_D)$ is unstable and therefore $\Omega_D$ is not a local minimizer of the energy. This means that there exist small volume preserving deformations of the spherical sector $\Omega_D$ which produce domains $\Omega_t$ and solutions $u_t$ of \eqref{eq:pde} in $\Omega_t$ whose energy $J(u_t)$ is smaller than the energy $J(u_D)$ of the positive radial solution $u_D$ in the spherical sector $\Omega_D$.

Moreover, observe that the function $f = f(s)$ could satisfy suitable hypotheses such that problem \eqref{eq:pde} has a unique positive solution $u_\Omega$ in any domain $\Omega \subset \Sigma_D$ (or more generally in $\Omega \subset \mathcal C$). This is the case, for example, when $f \equiv 1$, i.e., \eqref{eq:pde} is a ``relative" torsion problem. Then the energy functional $T(\Omega) = J(u_\Omega)$ is well defined for any domain $\Omega \subset \Sigma_D$. Hence we may ask whether a global minimum for $T$ exists, once the volume of $\Omega$ is fixed, and is given by the spherical sector $\Omega_D$. This question has been addressed in \cite{PacellaTralli2020}, \cite{PacellaTralli2021} and \cite{IacopettiPacellaWeth2022} when $f \equiv 1$, showing that $\Omega_D$ is a global minimizer if $\Sigma_D$ is a convex cone (\cite{PacellaTralli2021}), as a consequence of an isoperimetric inequality introduced in \cite{LionsPacella1990}, see also \cite{CabreRosOtonSerra2016, FigalliIndrei2013, RitoreRosales2004}. Instead, in \cite{IacopettiPacellaWeth2022} it is proved that $\Omega_D$ is not a local minimizer whenever $\lambda_1(D) < N - 1$, which is the same threshold we get in Theorem \ref{thm:stability_cone} for general nonlinearities.\\

The other example of an unbounded domain we consider in the present paper is a half-cylinder, defined as
\begin{equation}
\label{eq:cylinder_def}
\Sigma_\omega \coloneqq \omega \times (0, + \infty) \subset \mathbb R^N, 
\end{equation}
where $\omega \subset \mathbb R^{N - 1}$ is a smooth bounded domain. We denote the points in $\Sigma_\omega$ by $x = (x', x_N)$, $x' \in \omega$. In this case, a geometrically simple domain we consider is the bounded cylinder 
\begin{equation}
\label{eq:cylinder_bounded_def}
\Omega_\omega \coloneqq \{(x', x_N) \in \mathbb R^{N - 1} \ : \ x' \in \omega, \ 0 < x_N < 1 \}.
\end{equation}
In $\Omega_\omega$ we consider a positive solution 
\begin{equation}
\label{eq:extension_of_one_dimensional_solution}
u_\omega(x) = u_\omega(x_N)
\end{equation}
which is obtained by trivially extending to $\Omega_\omega$ a positive one-dimensional solution of the problem
\begin{equation}
\label{eq:nonlinear_ode_cylinder_intro}
\begin{cases}
- u'' = f(u) \quad \text{ in } (0, 1) \\
u'(0) = u(1)=0
\end{cases}
\end{equation}
for a nonlinearity $f$ for which such a solution exists.

Before stating the results concerning the stability of the pair $(\Omega_\omega, u_\omega)$ we again consider an auxiliary eigenvalue problem (but not singular):
\begin{equation}
\label{eq:nonsingular_eigenvalue_problem_intro}
\begin{cases}
-z'' - f'(u_\omega) z = \alpha z \quad \text{ in } (0, 1) \\
z'(0) = z(1) = 0
\end{cases}
\end{equation}

The problem \eqref{eq:nonsingular_eigenvalue_problem_intro} is considered in Section \ref{sec:cylinder} to study the spectrum of the linearized operator $- \Delta - f'(u_\omega)$. We denote by $\alpha_1$ the first eigenvalue of \eqref{eq:nonsingular_eigenvalue_problem_intro}.

We start by stating a sharp stability/instability result for the torsion problem, i.e., taking $f \equiv 1$ in \eqref{eq:pde}.

\begin{theorem}
    \label{thm:sharp_stability_for_torsion}
    Let $\Sigma_\omega \subset \mathbb R^N$, $N \geq 2$, and $\Omega_\omega$ be respectively, as in \eqref{eq:cylinder_def} and \eqref{eq:cylinder_bounded_def}, and let $u_\omega$ be the one-dimensional positive solution of \eqref{eq:pde} in $\Omega_\omega$ obtained by \eqref{eq:nonlinear_ode_cylinder_intro} for $f \equiv 1$. Let $\lambda_1(\omega)$ be the first nontrivial Neumann eigenvalue of the Laplace operator $- \Delta_{\mathbb R^{N - 1}}$ in the domain $\omega \subset \mathbb R^{N - 1}$. Then there exists a number $\beta \approx 1,439$ such that
    \begin{enumerate}[label=(\roman*)]
        \item if $\lambda_1(\omega) < \beta$, then the pair $(\Omega_\omega, u_\omega)$ is an unstable energy-stationary pair;

        \item if $\lambda_1(\omega) > \beta$, then the pair $(\Omega_\omega, u_\omega)$ is a stable energy-stationary pair.
    \end{enumerate}
\end{theorem}

Note that the number $\beta$ that gives the threshold for the stability is independent of the dimension $N$. Its value is obtained by solving numerically the equation $\sqrt{\lambda_1} \tanh(\sqrt{\lambda_1}) - 1 = 0$ (see \eqref{eq:Psiteo12} in the proof of Theorem \ref{thm:sharp_stability_for_torsion}).

It is interesting to observe that the instability result of Theorem \ref{thm:sharp_stability_for_torsion} is related to a bifurcation theorem obtained in \cite{FallMinlendWeth2017}. Indeed, if we consider the cylinder $\Sigma_\omega$ in $\mathbb R^2$, in which case $\omega$ is simply an interval in $\mathbb R$ and $\Omega_\omega$ is a rectangle, a byproduct of Theorem 1.1 of \cite{FallMinlendWeth2017} is the existence of a domain $\widetilde \Omega_\omega$ in $\Sigma_\omega$ that is a small deformation of the rectangle $\Omega_\omega$ and in which the overdetermined problem
\begin{equation*}
    \left\{
    \begin{array}{rcll}
        - \Delta u & = & 1 & \quad \text{ in } \widetilde \Omega_\omega \\
        u & = & 0 & \quad \text{ on } \Gamma_{\widetilde \Omega_\omega} \\
        \frac{\partial u}{\partial \nu} & = & c < 0 & \quad \text{ on } \Gamma_{\widetilde \Omega_\omega} \\ [3pt]
        \frac{\partial u}{\partial \nu} & = & 0 & \quad \text{ on } \Gamma_{1, \widetilde \Omega_\omega}
    \end{array}
    \right.
\end{equation*}
has a solution.

By looking at the proof of \cite{FallMinlendWeth2017} and relating it to our instability result it is clear that the bifurcation should occur when the eigenvalue $\lambda_1(\omega)$ crosses the value $\beta$ provided by Theorem \ref{thm:sharp_stability_for_torsion}.

The proof of Theorem \ref{thm:sharp_stability_for_torsion} can be derived from a general condition for the stability of the pair $(\Omega_\omega, u_\omega)$ in the nonlinear case, which is obtained in Theorem \ref{thm:generalinststab}. The proof of Theorem \ref{thm:generalinststab} involves auxiliary functions that appear naturally in the study of derivatives of the energy functional $T$, see Section \ref{sec:cylinder}.

Let us remark that in the case when $f \equiv 1$ we succeed in obtaining the sharp bound of Theorem \ref{thm:sharp_stability_for_torsion} because the solution given by \eqref{eq:extension_of_one_dimensional_solution} and \eqref{eq:nonlinear_ode_cylinder_intro} is explicit:
$$
u_\omega(x) = u_\omega(x_N) = \frac{1 - x_N^2}{2},
$$
and so are the auxiliary functions which are solutions of simple linear ODEs. This allows us to use the condition of Theorem \ref{thm:generalinststab} to obtain Theorem \ref{thm:sharp_stability_for_torsion}.

The result of Theorem \ref{thm:sharp_stability_for_torsion} gives a striking difference between the torsional energy problem and the isoperimetric problem in cylinders. Indeed, Proposition 2.1 of \cite{AfonsoIacopettiPacella2022} shows that the only stationary cartesian graphs for the perimeter functional are the flat ones. Instead, Theorem \ref{thm:sharp_stability_for_torsion} (as well as the result of \cite{FallMinlendWeth2017}) indicate that there are domains for which the overdetermined problem relative to \eqref{eq:pde}, with $f \equiv 1$, has a solution and whose relative boundary is a non-flat cartesian graph.

For the semilinear problem, we obtain a stability result for a large class of nonlinearities as soon as the eigenvalue $\lambda_1(\omega)$ is sufficiently large. Indeed, we have

\begin{theorem}
    \label{thm:stability_semilinear_cylinder}
    Let $\Sigma_\omega$ and $\Omega_\omega$ be as in \eqref{eq:cylinder_def} and \eqref{eq:cylinder_bounded_def}, and let $u_\omega$ be a positive one-dimensional solution of \eqref{eq:pde} in $\Omega_\omega$. Let $\alpha_1$ be the first eigenvalue of \eqref{eq:nonsingular_eigenvalue_problem_intro} and let $\lambda_1(\omega)$ be as in Theorem \ref{thm:sharp_stability_for_torsion}. If the nonlinearity $f$ satisfies $f(0) = 0$ and
    \begin{equation}
        \label{eq:condition_for_stability_semilinear_cylinder}
        \lambda_1(\omega) > \max\{- \alpha_1, \|f'(u_\omega)\|_\infty\},
    \end{equation}
    then the pair $(\Omega_\omega, u_\omega)$ is a stable energy-stationary pair.
\end{theorem}

The condition \eqref{eq:condition_for_stability_semilinear_cylinder} shows that the stability depends on an interplay between the geometry of the cylinder $\Sigma_\omega$ (through the eigenvalue $\lambda_1(\omega)$) and the nonlinearity $f$. On the contrary, numerical evidence shows, for the Lane-Emden nonlinearity \eqref{eq:LaneEmden}, that, if $\lambda_1$ is sufficiently close to $-\alpha_1$, instability occurs, see Remark \ref{rem:remark_on_numerics}.

Concerning the eigenvalue $\alpha_1$ in the bound \eqref{eq:condition_for_stability_semilinear_cylinder}, as well as the analogous one, $\lambda_1(D) > - \widehat \nu_1$, of Theorem \ref{thm:stability_cone}, we point out that they are used in the proofs of both theorems to deduce the positivity of some auxiliary functions. It is an open problem to understand if they really play a role in the stability/instability result.

We delay further comments on the results and their proofs to the respective sections.

The paper is organized as follows. In Section \ref{sec:general_unbounded_domains} we study problem \eqref{eq:pde} in domains $\Omega$ contained in a general unbounded set $\mathcal C$. We define the energy functional and its derivative with respect to variations of $\Omega$ which leave $\mathcal C$ invariant and preserve the measure of $\Omega$. This is done by considering nondegenerate solutions of \eqref{eq:pde} in $\Omega$.

In Section \ref{sec:cone} we consider the case when $\mathcal C$ is a cone $\Sigma_D$. In this setting we take domains which are defined by smooth radial graphs over $D$, in particular we consider the spherical sector $\Omega_D$ and a corresponding radial solution $u_D$ for which we prove the stability/instability result. 

Finally in Section \ref{sec:cylinder} we study the case of the cylinder $\Sigma_\omega$ and prove the corresponding stability/instability  result for the pair $(\Omega_\omega, u_\omega)$ when $\Omega_\omega$ is a bounded cylinder and $u_\omega$ is as in \eqref{eq:extension_of_one_dimensional_solution} and \eqref{eq:nonlinear_ode_cylinder_intro}.

\section{Semilinear elliptic problems in unbounded sets}
\label{sec:general_unbounded_domains}

In this section we consider problem \eqref{eq:pde} in a bounded Lipschitz domain $\Omega$ contained in an unbounded open set $\mathcal C$ which we assume to be (uniformly) Lipschitz regular.

Starting from a positive nondegenerate solution of \eqref{eq:pde} in $\Omega$ we show how to define an energy functional for small variations of $\Omega$ which preserve the volume.

\subsection{Nondegenerate solutions}
\label{subsec:nondegenerate_solutions}

Let $\Omega \subset \mathcal C$ be a bounded domain whose relative boundary $\Gamma_\Omega = \partial \Omega \cap \mathcal C$ is a smooth manifold (with boundary). As in Section \ref{sec:introduction} we set $\Gamma_{1, \Omega} = \partial \Omega\setminus{\overline{\Gamma}_\Omega}$.

We consider a positive weak solution $u_\Omega$ of \eqref{eq:pde} in the Sobolev space $H_0^1(\Omega \cup \Gamma_{1, \Omega})$, which is the subspace of $H^1(\Omega)$ of functions whose trace vanishes on $\Gamma_\Omega$. By standard variational methods, such as constrained minimization, Mountain-Pass Theorem etc, it is easy to exhibit many nonlinearities $f = f(s)$ for which such a solution exists. Moreover, with suitable assumptions on the growth of $f$ we also have, by regularity results, that $u_\Omega$ is a classical solution of \eqref{eq:pde} inside $\Omega$ and at any regular point of $\partial \Omega$, and that $u_\Omega$ is bounded (see also \cite[Proposition 3.1]{CiraoloPacellaPolvara2023}).

We assume that $u_\Omega$ is nondegenerate, i.e., the linearized operator
\begin{equation}
\label{eq:linearized_operator}
L_{u_\Omega} = - \Delta - f'(u_\Omega)
\end{equation}
does not have zero as an eigenvalue in $H_0^1(\Omega \cup \Gamma_{1, \Omega})$ or, in other words, $L_{u_\Omega}$ defines an isomorphism between $H_0^1(\Omega \cup \Gamma_{1, \Omega})$ and its dual space.
We consider small deformations of $\Omega$ which leave $\mathcal{C}$ invariant and would like to show that the nondegeneracy of $u_\Omega$ induces a local uniqueness result for solutions of \eqref{eq:pde} in the deformed domains. Thus we take a one-parameter family of diffeomorphisms $\xi_t$, for $t \in (- \eta, \eta)$, $\eta > 0$, associated to a smooth vector field $V$ such that $V(x) \in T_x\partial \mathcal C$ for every $x \in \partial \mathcal C^{\mathrm{reg}}$, $V(x)=0$ for $x\in\partial \mathcal C\setminus\partial \mathcal C^{\mathrm{reg}}$, and set $\Omega_t:= \xi_t(\Omega)$, where $T_x\partial \mathcal C$ denotes the tangent space to $\partial \mathcal C$ at the point $x$, and $\partial\mathcal C^{\mathrm{reg}}$ denotes the regular part of $\partial\mathcal C$. In particular $\Omega_0=\Omega$ and in order to simplify the notations we set
\begin{equation}
\Gamma_t \coloneqq \Gamma_{\Omega_t}, \quad \Gamma_{1, t} \coloneqq \Gamma_{1, \Omega_t}.
\end{equation}

\begin{prop}
\label{prop:local_uniqueness}
Let $u_\Omega$ be a positive nondegenerate solution of \eqref{eq:pde}, which belongs to $W^{1, \infty}(\Omega) \cap W^{2, 2}(\Omega)$. Let $V$ be a smooth vector field and let $\xi_t$ be the associated family of diffeomorphisms. Then there exists $\delta > 0$ such that for any $t \in (- \delta, \delta)$ there is a unique solution $u_t$ of the problem
\begin{equation}
\label{eq:pde_in_Omega_t}
\left\{
\begin{array}{rcll}
- \Delta u & = & f(u) & \quad \text{ in } \Omega_t \\[4pt]
u & = & 0 & \quad \text{ on } \Gamma_t \\[2pt]
\displaystyle \frac{\partial u}{\partial \nu} & = & 0 & \quad \text{ on } \Gamma_{1, t}
\end{array}
\right.
\end{equation}
in a neighborhood of the function $u_\Omega \circ \xi_t^{-1}$ in the space $H_0^1(\Omega_t \cup \Gamma_{1, t})$. Moreover, the map $t \mapsto u_t$ is differentiable.
\end{prop}

\begin{proof}
By using the diffeomorphism $\xi_t$ we can pass from the space $H_0^1(\Omega \cup \Gamma_{1, \Omega})$ to the space $H_0^1(\Omega_t \cup \Gamma_{1, t})$. Indeed, 
\begin{equation}
H_0^1(\Omega \cup \Gamma_{1, \Omega}) = \{v \circ \xi_t  \ : \ v \in H_0^1(\Omega_t \cup \Gamma_{1, t})\}.
\end{equation}
Moreover, $u_t$ is a weak solution of \eqref{eq:pde_in_Omega_t}, i.e.,
\begin{equation*}
\int_{\Omega_t} \nabla u_t \cdot \nabla v \ dx - \int_{\Omega_t} f(u_t) v \ dx = 0 \quad \forall v \in H_0^1(\Omega_t \cup \Gamma_{1, t})
\end{equation*}
if and only if the function $\widehat u_t = u_t \circ \xi_t \in H_0^1(\Omega \cup \Gamma_{1, \Omega})$ satisfies 
\begin{equation}
\label{eq:equation_for_u_t_hat}
\int_\Omega (M_t \nabla \widehat u_t)\cdot \nabla w J_t \ dx - \int_\Omega f(\widehat u_t) w J_t \ dx = 0 \quad \forall w \in H_0^1(\Omega \cup \Gamma_{1, \Omega})
\end{equation}
where 
\begin{equation*}
J_t(x) = \left|\det(\jac \xi_t(x))\right|
\end{equation*}
and 
\begin{equation}
M_t = [\jac \xi_t^{-1}(\xi_t(x))][\jac \xi_t^{-1}(\xi_t(x))]^T.
\end{equation}
In other words, setting $\widehat M_t:=M_t J_t$, we have that $\widehat u_t$ is a solution of 
\begin{equation*}
- \divergence (\widehat M_t \nabla \widehat u_t) - f(\widehat u_t)J_t = 0
\end{equation*}
in the space $H_0^1(\Omega \cup \Gamma_{1, \Omega})$.
Now we consider the map 
\begin{equation*}
\mathcal F: (- \eta, \eta) \times H_0^1(\Omega \cup \Gamma_{1, \Omega}) \to H_0^1(\Omega \cup \Gamma_{1, \Omega})^*
\end{equation*}
defined as
\begin{equation}
\mathcal F(t, v) = - \divergence(\widehat M_t \nabla v) - f(v) J_t.
\end{equation}
Since $u_\Omega$ is a solution in $\Omega$ and $\xi_0$ is the identity map we have
\begin{equation*}
\mathcal F(0, u_\Omega) = 0.
\end{equation*}

Notice that $\mathcal F$ is differentiable with respect to to $v$, and
\begin{equation}
\label{eq:partial_v_mathcal_F}
\partial_v \mathcal F(0, u_\Omega) = - \Delta - f'(u_\Omega).
\end{equation}
Indeed, for any $h \in H_0^1(\Omega \cup \Gamma_{1, \Omega})$ we have
\begin{align}
\frac{\mathcal F(t, v + \varepsilon h) - \mathcal F(t, v)}{\varepsilon}
& = \frac{- \divergence (\widehat M_t(\nabla v + \varepsilon \nabla h) ) - f(v + \varepsilon h)J_t - (-\divergence (\widehat M_t \nabla v) - f(v)J_t)}{\varepsilon} \nonumber \\
& = - \frac{\divergence (\varepsilon \widehat M_t \nabla h)}{\varepsilon} - \frac{(f(v + \varepsilon h) - f(v))J_t}{\varepsilon} \nonumber \\
& \to - \divergence(\widehat M_t \nabla h) - f'(v)J_t \nonumber 
\end{align}
as $\varepsilon \to 0$. Hence $\mathcal F$ is differentiable and evaluating $\partial_v \mathcal F$ at $(0, u_\Omega)$ we obtain \eqref{eq:partial_v_mathcal_F}.

By the nondegeneracy assumption on the solution $u_\Omega$, we infer that \eqref{eq:partial_v_mathcal_F} gives an isomorphism between $H_0^1(\Omega \cup \Gamma_{1, \Omega})$ and $H_0^1(\Omega \cup \Gamma_{1, \Omega})^*$. Then, by the Implicit Function Theorem, there exists an interval $(- \delta, \delta)$ and a neighborhood $\mathcal B$ of $u_\Omega$ in $H_0^1(\Omega \cup \Gamma_{1, \Omega})$ such that for every $t \in (- \delta, \delta)$ there exists a unique function $\widehat u_t \in H_0^1(\Omega \cup \Gamma_{1, \Omega})$ in $\mathcal B$ such that $\mathcal F(t, \widehat u_t) = 0$, that is, $\widehat u_t$ is the unique solution (in $\mathcal B$) of \eqref{eq:equation_for_u_t_hat}. It follows that $u_t = \widehat u_t \circ \xi_t^{-1}$ is the unique solution of \eqref{eq:pde_in_Omega_t} in a neighborhood of $u_\Omega \circ \xi_t^{-1}$ in $H_0^1(\Omega_t \cup \Gamma_{1, t})$.

Finally, since the map $t \mapsto \widehat u_t$ is smooth, so is the map $t \mapsto u_t$. In addition
\begin{equation}
\label{eq:def_u_tilde}
\widetilde u \coloneqq \left. \frac{d}{dt} \right|_{t = 0} u_t = \left(\left. \frac{d}{dt} \right|_{t = 0}\widehat u_t\right) -  \langle \nabla u_\Omega, V\rangle.
\end{equation}

The proof is complete.
\end{proof}

Note that, as for $u_\Omega$, $u_t$ is a classical solution of \eqref{eq:pde_in_Omega_t} in $\Omega_t$ and on the regular part of $\partial \Omega_t$. By Proposition \ref{prop:local_uniqueness} we have that the energy functional
\begin{equation}
\label{eq:def_T}
T(\Omega_t) = J(u_t) = \frac{1}{2} \int_{\Omega_t} |\nabla u_t|^2 \ dx - \int_{\Omega_t} F(u_t) \ dx,
\end{equation}
where $F(s) = \int_0^s f(\tau) \ d\tau$, is well defined for all sufficiently small $t$. Observe that, since $u_t$ is a solution to \eqref{eq:pde_in_Omega_t}, we have 
\begin{equation*}
\int_{\Omega_t} |\nabla u_t|^2 \ dx = \int_{\Omega_t} f(u_t)u_t \ dx,
\end{equation*}
so we can also write
\begin{equation}
\label{eq:def_T_alternative}
T(\Omega_t) = \frac{1}{2} \int_{\Omega_t} f(u_t)u_t \ dx - \int_{\Omega_t} F(u_t) \ dx.
\end{equation}

In the next result we show that $T$ is differentiable with respect to $t$ and compute its derivative at $t = 0$, that is, at the initial domain $\Omega$.

\begin{prop}
\label{prop:T_1_general}
Assume that $u_\Omega$ is a positive nondegenerate solution of \eqref{eq:pde} which belongs to $W^{1, \infty}(\Omega) \cap W^{2, 2}(\Omega)$. Then
\begin{equation}
\label{eq:T_1_general} 
\left. \frac{d}{dt} \right|_{t = 0} T(\Omega_t) = - \frac{1}{2} \int_{\Gamma_\Omega} |\nabla u_\Omega|^2 \langle V, \nu \rangle \ d\sigma.
\end{equation}
\end{prop}

\begin{proof}
Recall from Proposition \ref{prop:local_uniqueness} that $t \mapsto u_t$ is smooth and \eqref{eq:def_u_tilde} holds. Differentiating the equation $- \Delta u_t = f(u_t)$ with respect to $t$ we obtain
\begin{equation}\label{eq:equazionetildeu}
- \Delta \widetilde u = f'(u_\Omega) \widetilde u \quad \text{ in } \Omega.
\end{equation}
Now observe that by the hypotheses on $u_\Omega$ we have that
\begin{equation}
\widetilde u + \langle \nabla u_\Omega, V \rangle = \left(\left. \frac{d}{dt} \right|_{t = 0}  \widehat u_t\right) \in H_0^1(\Omega \cup \Gamma_{1,\Omega}),
\end{equation}
thus
\begin{equation}\label{eq:bondcondtildeugammaomega}
\widetilde u = - \frac{\partial u_\Omega}{\partial \nu} \langle V, \nu \rangle \quad \text{ on } \Gamma_\Omega.
\end{equation}
Finally, since $\xi_t$ maps $\partial \mathcal C$ into itself we have that, for all small $t$ and $x \in(\partial\mathcal C\cap\partial\Omega)^{\mathrm{reg}}$
\begin{equation*}
\langle \nabla u_t(\xi_t(x)), \nu(\xi_t(x)) \rangle = 0.
\end{equation*}
Differentiating this relation with respect to $t$ and evaluating at $t=0$ we obtain
\begin{equation*}
0 = \langle \nabla \widetilde u(x), \nu(x) \rangle + d_x(\langle \nabla u_\Omega, \nu\rangle)[V(x)],
\end{equation*}
where $d_x(\langle \nabla u_\Omega, \nu\rangle)[V(x) ]$ is the differential of the function $\langle \nabla u_\Omega, \nu\rangle\big|_{(\partial\mathcal C\cap\partial\Omega)^{\mathrm{reg}}}$ computed at $x$, along $V(x)$. Then, since $\langle \nabla u_\Omega, \nu\rangle=0$ on $(\partial\mathcal C\cap\partial\Omega)^{\mathrm{reg}}$, and in view of \eqref{eq:equazionetildeu}, \eqref{eq:bondcondtildeugammaomega}, we infer that  $\widetilde u$ satisfies
\begin{equation}
\label{eq:equation_for_u_tilde}
\left\{
\begin{array}{rcll}
- \Delta \widetilde u & = & f'(u_\Omega) \widetilde u & \quad \text{ in } \Omega \\[4pt]
\widetilde u & = & \displaystyle - \frac{\partial u_\Omega}{\partial \nu} \langle V, \nu \rangle & \quad \text{ on } \Gamma_\Omega \\[2pt]
\displaystyle \frac{\partial \widetilde u}{\partial \nu} & = & 0 & \quad \text{ on }\ \Gamma_{1,\Omega}
\end{array}
\right.
\end{equation}
in the classical sense in the interior of $\Omega$ and on the regular part of $\partial \Omega$.

Recalling \eqref{eq:def_T_alternative} we can write
\begin{equation*}
T(\Omega_t) = \int_{\Omega_t} \frac{1}{2} \left(f(u_t)u_t - F(u_t) \right)\ dx.
\end{equation*}

Since $t\mapsto f(u_t)u_t - F(u_t)$ is differentiable at $t=0$, $\partial\Omega$ is Lipschitz and taking into account that $u_\Omega \in W^{1, \infty}(\Omega) \cap W^{2, 2}(\Omega)$, then, applying \cite[Theorem 5.2.2]{HenrotPierre2018}, we can compute the derivative with respect to $t$ of the functional $T$ obtaining that 
\begin{align}
\left. \frac{d}{dt}\right|_{t = 0} T(\Omega_t)
& = \frac{1}{2} \int_\Omega (f'(u_\Omega)\widetilde u u_\Omega + f(u_\Omega) \widetilde u) \ dx - \int_\Omega f(u_\Omega) \widetilde u \ dx  \nonumber \\
& \quad + \int_{\partial \Omega} \left( \frac{1}{2}f(u_\Omega)u_\Omega - F(u_\Omega)\right) \langle V, \nu \rangle \ d \sigma \nonumber \\
& = \frac{1}{2} \int_\Omega \left(f'(u_\Omega) \widetilde u u_\Omega - f(u_\Omega)\widetilde u \right)\ dx \nonumber \\
& = \frac{1}{2} \int_\Omega \left( (- \Delta \widetilde u) u_\Omega + \Delta u_\Omega \widetilde u \right) \ dx \nonumber \\
& = \frac{1}{2} \int_{\partial \Omega} \left( \widetilde u \frac{\partial u_\Omega}{\partial \nu} - u_\Omega \frac{\partial \widetilde u}{\partial \nu} \right) \ d \sigma \nonumber \\
& = - \frac{1}{2} \int_{\Gamma_\Omega} |\nabla u_\Omega|^2 \langle V, \nu \rangle \ d \sigma.
\end{align}
The previous applications of the Divergence Theorem are justified by arguing as in \cite[Lemma 2.1]{PacellaTralli2020}, where the regularity hypothesis on $u_\Omega$ comes into play.
\end{proof}

\begin{remark}
It is not difficult to see that $\widetilde u$ is also a weak solution of \eqref{eq:equation_for_u_tilde}. Indeed, let $\varphi \in C_c^\infty(\Omega \cup \Gamma_{1,\Omega})$. Then, for all sufficiently small $t$, we also have $\varphi \in C_c^\infty(\Omega_t \cup \Gamma_{1, t})$. Hence, since $u_t$ is a weak solution to \eqref{eq:pde_in_Omega_t}, we have
\begin{equation}
\label{eq:u_tilde_is_a_weak_solution}
0 = \int_{\Omega_t} \nabla u_t \nabla \varphi \ dx - \int_{\Omega_t} f(u_t) \varphi \ dx = \int_{\Omega} \nabla u_t \nabla \varphi \ dx - \int_{\Omega} f(u_t) \varphi \ dx.
\end{equation}
Now, as proved in \cite[Claim (3.17)]{IacopettiPacellaWeth2022}, it holds that
\begin{equation*}
\left. \frac{d}{dt} \right|_{t = 0} \nabla u_t = \nabla \widetilde u.
\end{equation*}
Then, taking the derivative with respect to $t$ in \eqref{eq:u_tilde_is_a_weak_solution}, evaluating at $t=0$, and since $\varphi$ is arbitrary, we easily conclude.
\end{remark}

Let us now consider domains $\Omega \subset \mathcal C$ of fixed measure $c>0$ and define
\begin{equation}\label{eq:admissibledomains}
\mathcal A \coloneqq \{\Omega \subset \mathcal C \ : \ \Omega \text{ is admissible and } |\Omega| = c\},
\end{equation}
where admissible means that $\Omega \subset \mathcal C$ is a bounded domain with smooth relative boundary $\Gamma_\Omega \coloneqq \partial \Omega \cap \mathcal C$, $\partial\Gamma_\Omega$ is a smooth $(N-2)$-dimensional manifold and $\Gamma_{1, \Omega} \coloneqq \partial \Omega \setminus \overline \Gamma_\Omega$ is such that $\mathcal H^{N - 1}(\Gamma_{1, \Omega}) > 0$.   We consider vector fields that induce deformations that preserve the volume. More precisely we take a one-parameter family of diffeomorphisms $\xi_t$, $t\in (-\eta,\eta)$, associated to a smooth vector field $V$ such that $V(x) \in T_x\partial \mathcal C^\mathrm{reg}$ for all $x \in \partial \mathcal C^\mathrm{reg}$, and satisfying the condition $|\Omega_t| = |\Omega|$, for all $t\in (-\eta,\eta)$, where $\Omega_t = \xi_t(\Omega)$.

\begin{definition}
\label{def:energy_stationary}
We say that the pair $(\Omega, u_\Omega)$ is energy-stationary under a volume constraint if
\begin{equation}
\left.\frac{d}{dt}\right|_{t = 0} T(\Omega_t) = 0
\end{equation}
for any vector field tangent to $\partial \mathcal C$ such that the associated one-parameter family of diffeomorphisms preserves the volume.
\end{definition}

A characterization of energy-stationary pairs in $\mathcal C$ is the following:

\begin{prop}
\label{prop:energy_stationary_char}
Let $\Omega \in \mathcal{A}$ and assume that $u_\Omega \in W^{1, \infty}(\Omega) \cap W^{2, 2}(\Omega)$ is a  nondegenerate positive solution of \eqref{eq:pde}. Then $(\Omega, u_\Omega)$ is energy-stationary under a volume constraint if and only if $u_\Omega$ satisfies the overdetermined condition $|\nabla u_\Omega| = \text{constant}$ on $\Gamma_\Omega$.
\end{prop}

\begin{proof}
Let $\xi_t$ be an arbitrary admissible one-parameter family of diffeomorphisms and let $V$ be the associated vector field. Since the volume is preserved and $V(x) \in T_x\partial \mathcal C$ on $\partial \mathcal C$, 
\begin{equation}
\label{eq:Pacella_Tralli_2_4.6}
0 = \left. \frac{d}{dt} \right|_{t = 0} |\Omega_t| = \int_{\partial \Omega} \langle V, \nu \rangle \ d \sigma = \int_{\Gamma_{\Omega}} \langle V, \nu \rangle \ d \sigma.
\end{equation}
If $|\nabla u_\Omega|$ is constant on $\Gamma_{\Omega}$, then $(\Omega, u_\Omega)$ is energy-stationary, in view of \eqref{eq:T_1_general} and \eqref{eq:Pacella_Tralli_2_4.6}. On the other hand, if $(\Omega, u_\Omega)$ is energy stationary, then 
\begin{equation}
\int_{\Gamma_{\Omega}} (|\nabla u_\Omega|^2 - a) \langle V, \nu\rangle \ d \sigma = 0
\end{equation}
for every constant $a$ and every admissible vector field $V$. Assume by contradiction that $|\nabla u_\Omega|$ is not constant on $\Gamma_{\Omega}$. Then there exists a compact set $K \subset \Gamma_{\Omega}$, with nonempty interior part, such that $|\nabla u_\Omega|$ is not constant on $K$. Take a nonnegative cutoff function $\Theta$ such that $\Theta \equiv 1$ in $K$, and choose
\begin{equation}
a = \frac{\int_{\Gamma_\Omega} \Theta |\nabla u_\Omega|^2 \ d\sigma}{\int_{\Gamma_\Omega} \Theta \ d\sigma}.
\end{equation}
Then we can build a deformation from the vector field $V = (|\nabla u_\Omega|^2 - a) \Theta \nu$, and in this case, since $(\Omega, u_\Omega)$ is energy stationary, we would have
\begin{equation}
\int_K (|\nabla u_\Omega|^2 - a)^2 \ d \sigma = 0,
\end{equation}
which contradicts the fact that $|\nabla u_\Omega|$ is not constant on $K$. The proof is complete.
\end{proof}

\begin{remark}
\label{rem:remark_RN_energy_stationary}
It is relevant to observe that all concepts introduced in this section apply to the case when $\Gamma_{1, \Omega}$ is empty, or, equivalently, when $\mathcal C = \mathbb R^N$. Thus all the above results hold for Dirichlet problems in domains in the whole space. In this case it is known, by Serrin's Theorem (see \cite{Serrin1971}) that if a positive solution for the overdetermined problem
\begin{equation}
\left\{
\begin{array}{rcll}
- \Delta u & = & f(u) & \quad \text{ in } \Omega \\[4pt]
u & = & 0 & \quad \text{ on } \partial \Omega \\[2pt]
\displaystyle \frac{\partial u}{\partial \nu} & = & \text{constant} & \quad \text{ on } \partial \Omega  
\end{array}
\right.
\end{equation}
exists, then $\Omega$ is a ball. Therefore, in view of Proposition \ref{prop:energy_stationary_char}, it follows that the only energy-stationary pairs in $\mathbb R^N$ are $(B, u_B)$, where $B$ is a ball and $u_B$ is a nondegenerate positive solution.
\end{remark}

\begin{remark}
We observe that all the results in this section hold true also for non-degenerate sign-changing solutions $u_\Omega$ to \eqref{eq:pde}. However, since in the sequel we study the stability in the case of positive solutions, we have considered only this case
\end{remark}

\section{The case of the cone}
\label{sec:cone}

Let $D \subset \mathbb{S}^{N - 1}$ be a smooth domain on the unit sphere and let $\Sigma_D$ be the cone spanned by $D$, which is defined as
\begin{equation}
\label{eq:def_cone}
\Sigma_D \coloneqq \{x \in \mathbb R^N \ : \ x = t q, \ q \in D, \ t > 0\}.
\end{equation}

In $\Sigma_D$ we consider admissible domains $\Omega$, in the sense of \eqref{eq:admissibledomains}, that are strictly star-shaped with respect to the vertex of the cone, which we choose to be the origin $0$ in $\mathbb R^N$. In other words, we consider domains whose relative boundary is the radial graph in $\Sigma_D$ of a  function in $C^2(\overline D, \mathbb R)$. Hence for $\varphi \in C^2(\overline D, \mathbb R)$ we set
\begin{equation}
\label{eq:radial_graph}
\Gamma_\varphi \coloneqq \{x \in \mathbb R^N \ : \ x = e^{\varphi(q)}q, \ q \in D\}
\end{equation}
and consider the strictly star-shaped domain $\Omega_\varphi$ defined as 
\begin{equation}
\label{eq:star_shaped_domain_in_the_cone}
\Omega_\varphi \coloneqq \{x \in \mathbb R^N \ : \ x = tq, \ 0 < t < e^{\varphi(q)}, \ q \in D\}.
\end{equation}
To simplify the notation we set
\begin{equation*}
\Gamma_{1, \varphi} \coloneqq \Gamma_{1, \Omega_\varphi} = \partial \Omega_\varphi \setminus\overline\Gamma_\varphi.
\end{equation*}

\subsection{Energy functional for star-shaped domains}

In $\Omega_\varphi$ we consider the semilinear elliptic problem
\begin{equation}
\label{eq:pde_cone_starshaped}
\left\{
\begin{array}{rcll}
- \Delta u & = & f(u) & \quad \text{ in } \Omega_\varphi \\[4pt]
u & = & 0 & \quad \text{ on } \Gamma_\varphi \\[2pt]
\displaystyle \frac{\partial u}{\partial \nu} & = & 0 & \quad \text{ on } \Gamma_{1, \varphi}\setminus\{0\}
\end{array}
\right.
\end{equation}
and assume throughout this section that a bounded positive nondegenerate solution $u_{\Omega_\varphi}$ exists and belongs to $W^{1, \infty}(\Omega_\varphi) \cap W^{2, 2}(\Omega_\varphi)$. Then we can apply the results of Section \ref{sec:general_unbounded_domains} and define the energy functional $T$ as in \eqref{eq:def_T} for small variations of $\Omega_\varphi$. Since $\Omega_\varphi$ is strictly star-shaped, this property also holds for the domains obtained by small regular deformations. Thus it is convenient to parametrize the domains and their variations by $C^2$ functions defined on $\overline D$. Hence, for $v \in C^2(\overline D, \mathbb R)$ and $t \in (- \eta, \eta)$, where $\eta > 0$ is a fixed number sufficiently small, we consider the domain variations $\Omega_{\varphi + tv} \subset \Sigma_D$.

Let $\xi:(- \eta, \eta) \times \overline{\Sigma}_D\setminus\{0\} \to \overline{\Sigma}_D\setminus\{0\}$ be the map defined by 
\begin{equation}
\xi(t, x) = e^{t v\left(\frac{x}{|x|} \right)}x.
\end{equation}
Then $\xi|_{\Omega_\varphi}(t, \cdot) : \Omega_\varphi \to \Omega_{\varphi + tv}$ is a diffeomorphism, whose inverse is 
\begin{equation}
(\xi|_{\Omega_\varphi})^{-1}(t, x) = e^{-t v \left(\frac{x}{|x|} \right)}x = \xi(-t, x).
\end{equation}
By definition, $\xi(t, x) \in \partial \Sigma_D \setminus \{0\}$ for all $(t, x) \in (- \eta, \eta) \times (\partial \Sigma_D \setminus \{0\})$ and $\xi$ is the flow associated to the vector field
\begin{equation}\label{eq:defVradial}
V(x) = v\left(\frac{x}{|x|} \right)x,
\end{equation}
since $\xi(0, x) = x$ and 
\begin{equation*}
\frac{d}{dt} \xi(t, x) = e^{tv\left(\frac{x}{|x|} \right)} v\left(\frac{x}{|x|} \right)x = V(\xi(t, x)).
\end{equation*}
The energy functional $T$ in \eqref{eq:def_T} becomes a functional defined on functions in $C^2(\overline D, \mathbb R)$. More precisely, we define, for every $v \in C^2(\overline D, \mathbb R)$,
\begin{equation}\label{eq:Tincrement}
T(\varphi + tv) := T(\Omega_{\varphi + tv}) = J(u_{\varphi + tv}),
\end{equation}
for $t \in (- \delta, \delta)$ with $\delta > 0$ small, where $$u_{\varphi + tv} \coloneqq u_{\Omega_{\varphi + tv}}$$ is the unique positive solution of \eqref{eq:pde_cone_starshaped} in the domain $\Omega_{\varphi + tv}$, in a neighborhood of $u_\varphi \circ \xi(t, \cdot)^{-1}$. 

We now compute the first derivative of the functional $T$ at $\varphi$ along a direction $v \in C^2(\overline D, \mathbb R)$, i.e. the derivative with respect to $t$ of \eqref{eq:Tincrement} computed at $t=0$.

\begin{lemma}
\label{lem:T1_radial_graphs}
Let $\varphi \in C^2(\overline D, \mathbb R)$ and assume that $u_\varphi$ is a bounded positive nondegenerate solution to \eqref{eq:pde_cone_starshaped} and that $u_\varphi$ belongs to $W^{1, \infty}(\Omega_\varphi) \cap W^{2, 2}(\Omega_\varphi)$. Then for any $v \in C^2(\overline D, \mathbb R)$ it holds that
\begin{equation}
\label{eq:T1_radial_graphs}
T'(\varphi)[v] = - \frac{1}{2} \int_D \left(\frac{\partial u_\varphi}{\partial \nu}(e^\varphi q) \right)^2 e^{N \varphi} v \ d \sigma
\end{equation}
\end{lemma}

\begin{proof}
The result follows from Proposition \ref{prop:T_1_general}. Indeed, since the exterior unit normal to $\Gamma_\varphi$ is given by
$$
\nu (x)= \frac{\frac{x}{|x|} - \nabla_{\mathbb{S}^{N - 1}} \varphi \left(\frac{x}{|x|} \right)}{\sqrt{1 + \left|\nabla_{\mathbb{S}^{N - 1}} \varphi \left(\frac{x}{|x|} \right)\right|^2}},\ \quad x\in \Gamma_\varphi,
$$
where $\nabla_{\mathbb{S}^{N - 1}}$ is the gradient in $\mathbb{S}^{N - 1}$ (see \cite[Sect. 2]{IacopettiPacellaWeth2022}), then, from \eqref{eq:defVradial}, it follows that
\begin{equation*}
%\label{eq:V_times_nu}
\langle V, \nu \rangle = \frac{|x|}{\sqrt{1 + \left|\nabla_{\mathbb{S}^{N - 1}} \varphi \left(\frac{x}{|x|} \right)\right|^2}} v \left(\frac{x}{|x|} \right) \quad \text{on $\Gamma_\varphi$}.
\end{equation*}
Hence, using the parametrization $x = e^{\varphi(q)} q$, for $q \in D$, taking into account that the induced $(N-1)$-dimensional area element on $\Gamma_\varphi$ is given by
$$
d \sigma_{\Gamma_\varphi} = e^{(N - 1)\varphi}\sqrt{1 + |\nabla_{\mathbb{S}^{N - 1}} \varphi|^2} \ d \sigma,
$$
and since $u_\varphi=0$ on $\Gamma_\varphi$, then, from \eqref{eq:T_1_general}, we readily obtain \eqref{eq:T1_radial_graphs}.
\end{proof}

The next step is to compute the second derivative of $T$ at $\Omega_\varphi$ with respect to directions $v, w \in C^2(\overline D, \mathbb R)$

\begin{lemma}
\label{lem:T_2_cone}
Let $\varphi$ and $u_\varphi$ be as in Lemma \ref{lem:T1_radial_graphs}. Then for any $v, w \in C^2(\overline D, \mathbb R)$ it holds
\begin{align}
T''(\varphi)[v, w]
& = - \frac{N}{2} \int_D e^{N \varphi} v w \left(\frac{\partial u_\varphi}{\partial \nu}(e^\varphi q) \right)^2 \ d \sigma \nonumber \\
& \quad - \int_D e^{N \varphi} v \frac{\partial u_\varphi}{\partial \nu}(e^\varphi q) \frac{\partial \widetilde u_w}{\partial \nu}(e^\varphi q) \ d \sigma \nonumber \\
& \quad - \int_D e^{N \varphi} v w \frac{\partial u_\varphi}{\partial \nu}(e^\varphi q) (D^2u_\varphi(e^\varphi q) e^\varphi q) \cdot \nu \ d \sigma \nonumber \\
& \quad + \int_D e^{N \varphi} v \frac{\partial u_\varphi}{\partial \nu}(e^\varphi q) \frac{\nabla u_\varphi(e^\varphi q) \cdot \nabla_{\mathbb{S}^{N - 1}} w}{\sqrt{1 + |\nabla_{\mathbb{S}^{N - 1}} \varphi|^2}} \ d \sigma \nonumber \\
& \quad + \int_D e^{N \varphi} \left(\frac{\partial u_\varphi}{\partial \nu}(e^\varphi q) \right)^2 \frac{\nabla_{\mathbb{S}^{N - 1}} \varphi \cdot \nabla_{\mathbb{S}^{N - 1}} w}{1 + |\nabla_{\mathbb{S}^{N - 1}} \varphi|^2} \ d \sigma\label{eq:T2_radial_graphs},
\end{align}
where $\widetilde u_w = \left. \frac{d}{ds}\right|_{s = 0} u_{\varphi + sw}$ satisfies \eqref{eq:equation_for_u_tilde} with $V(x) = w\left(\frac{x}{|x|}\right)x$.
\end{lemma}

\begin{proof}
The proof is the same as that of \cite[Lemma 3.2]{IacopettiPacellaWeth2022} and therefore we omit it.
\end{proof}

In view of Definition \ref{def:energy_stationary}, we are interested in studying pairs $(\Omega_\varphi, u_\varphi)$ which are energy-stationary under a volume constraint. Thus we need to consider domains $\Omega_\varphi$ with a fixed volume. We recall that the volume of the domain defined by the radial graph of a function $\varphi \in C^2(\overline D, \mathbb R)$ is given by
\begin{equation*}
\mathcal V(\varphi) \coloneqq \mathcal V(\Omega_\varphi) = |\Omega_\varphi| = \frac{1}{N} \int_D e^{N\varphi} \ d\sigma.
\end{equation*}
Simple computations yield, for $v, w \in C^2(\overline D, \mathbb R)$:
\begin{equation}
\label{eq:volume_derivatives}
\mathcal V'(\varphi)[v] = \int_D e^{N \varphi} v\, \, d\sigma, \qquad \mathcal V''(\varphi)[v, w] = N \int_D e^{N\varphi} v w \ d \sigma.
\end{equation}
Then, for $c > 0$ we define the manifold 
\begin{equation}
\label{eq:def_M}
M \coloneqq \{\varphi \in C^2(\overline D, \mathbb R) \ : \ \mathcal V(\varphi) = c\},
\end{equation}
whose tangent space at any point $\varphi \in M$ is given by
\begin{equation*}
T_\varphi M = \left\{v \in C^2(\overline D, \mathbb R) \ : \ \int_D e^{N \varphi} v \ d\sigma = 0 \right\}.
\end{equation*}
We restrict the energy functional to the manifold $M$ and denote it by
\begin{equation*}
I(\varphi) := T\big|_M(\varphi).
\end{equation*}
Clearly, if the pair $(\Omega_\varphi, u_\varphi)$ is energy-stationary under a volume constraint, in the sense of Definition \ref{def:energy_stationary}, then $\varphi \in M$ is a critical point of $I$. Hence, by the Theorem of Lagrange multipliers, there exists $\mu \in \mathbb R$ such that
\begin{equation}
\label{eq:Lagrange_multiplier}
T'(\varphi) = \mu \mathcal V'(\varphi).
\end{equation}
Moreover, the following result holds true:
\begin{prop}
\label{prop:Lagrange_multiplier_is_negative}
Let $\varphi \in M$ such that $(\Omega_\varphi, u_\varphi)$ is energy-stationary under the volume constraint. Then the Lagrange multiplier $\mu$ is negative and
\begin{equation}
\frac{\partial u_\varphi}{\partial \nu} = - \sqrt{- 2 \mu} \quad \text{ on } \quad \Gamma_\varphi.
\end{equation} 
\end{prop}

\begin{proof}
The proof is the same as in \cite[Lemma 4.1]{IacopettiPacellaWeth2022}
\end{proof}

For the second derivative of $I$ we have
\begin{lemma}
\label{lem:Lagrange_multipliers_second_derivatives}
Let $\varphi \in M$ and let $v, w \in T_\varphi M$. If $(\Omega_\varphi, u_\varphi)$ is energy-stationary under the volume constraint, then 
\begin{equation}
\label{eq:Lagrange_multipliers_second_derivatives}
I''(\varphi)[v, w] = T''(\varphi)[v, w] - \mu \mathcal V''(\varphi)[v, w].
\end{equation}
\end{lemma}

\begin{proof}
The proof is the same as in \cite[Lemma 4.3]{IacopettiPacellaWeth2022}.
\end{proof}

\subsection{Spherical sectors and radial solutions}

Given a cone $\Sigma_D$ we consider the spherical sector $\Omega_D$ obtained by intersecting $\Sigma_D$ with the unit ball $B_1$. Obviously its relative boundary $\Gamma_{\Omega_D}$ is the radial graph obtained by taking $\varphi \equiv 0$ in \eqref{eq:radial_graph}, which is in fact the domain $D$ which spans the cone, that is $\Gamma_{\Omega_D} = D$.

In the spherical sector $\Omega_D$ we would like to consider a nondegenerate positive radial solution $u_D \coloneqq u_{\Omega_D}$ of \eqref{eq:pde_cone_starshaped}, hence we first recall conditions on the nonlinearity $f$ which ensure that a positive radial solution of \eqref{eq:pde_cone_starshaped} in $\Omega_D$ exists. Observe that such $u_D$ is just the restriction to $\Omega_D$ of a positive radial solution of the Dirichlet problem
\begin{equation}
\label{eq:Dirichlet_problem_in_the_ball}
\left\{
\begin{array}{rcll}
- \Delta u & = & f(u) & \quad \text{ in } B_1\\[3pt]
u & = & 0 & \quad \text{ on } \partial B_1
\end{array}
\right.
\end{equation}

\begin{prop}
\label{prop:existence_radial_sector}
Let $f : \mathbb R \to \mathbb R$ be a locally Lipschitz continuous function. Assume that $f$ satisfies one of the following:
\begin{enumerate}
\item $|f(s)| \leq a|s| + b$ for all $s > 0$, where $b > 0$ and $a \in (0, \mu_1(B_1))$, where $\mu_1(B_1)$ is the first eigenvalue of the operator $- \Delta$ in $H_0^1(B_1)$.

\item $f:[0, + \infty) \to [0, + \infty)$ is non-increasing.

\item
		\begin{itemize}
		\item $|f(s)| < c|s|^p + d$, where $c, d > 0$ and $p \in \left(1, \frac{N + 2}{N - 2} \right)$ if $N \geq 3$, $p > 1$ if $N = 2$;
		
		\item $f(s) = o(s)$ as $s \to 0$;
		
		\item There exist $\gamma > 2$, $\kappa > 0$ such that for $|s| > \kappa$ it holds
		$$
		0 < \gamma F(s) < s f(s);
		$$
		
		\item $\displaystyle f'(s) > \frac{f(s)}{s}$ for all $s > 0$.
		\end{itemize}
\end{enumerate}
Then a radial positive solution of \eqref{eq:Dirichlet_problem_in_the_ball} in $B_1$, and hence of \eqref{eq:pde_cone_starshaped} in $\Omega_D$, exists.
\end{prop}

\begin{proof}
In cases \textit{(i)} and \textit{(ii)}, the corresponding functional 
$$J(u) = \frac{1}{2}\int_{B_1} |\nabla u|^2 \ dx - \int_{B_1} F(u) \ dx
$$ 
is coercive and weakly lower semicontinuous in the space $H^1_{0, rad}(B_1)$, which is the subspace of $H_0^1(B_1)$ of radial functions, and so it has a minimum which is a solution of \eqref{eq:Dirichlet_problem_in_the_ball}. In the case \textit{(iii)} standard variational methods, such as minimization on the Nehari manifold or Mountain Pass type theorems give a positive solution of \eqref{eq:Dirichlet_problem_in_the_ball}, which is then radial by the Gidas-Ni-Nirenberg Theorem (see \cite{GidasNiNirenberg1979}). We refer to \cite{AmbrosettiMalchiodi2007} and \cite{DamascelliPacella2019} for the details.
\end{proof}

We point out that a radial solution $u_D$ is always a classical solution of \eqref{eq:Dirichlet_problem_in_the_ball} in $B_1$, and hence in $\Omega_D$. In particular, $u_D$ is bounded and $u_D \in C^2(\overline B_1)$

Now we would like to study the nondegeneracy of a radial solution $u_D$ of \eqref{eq:pde_cone_starshaped} in $\Omega_D$.

As recalled in Section \ref{subsec:nondegenerate_solutions}, we need conditions that ensure that zero is not an eigenvalue of the linearized operator
\begin{equation}
\label{eq:linearized_operator_radial_sector}
L_{u_D} = - \Delta - f'(u_D)
\end{equation}
in the space $H_0^1(\Omega_D \cup \Gamma_{1, 0})$, where $\Gamma_{1, 0} = \partial \Omega_D\setminus\overline{\Gamma}_{\Omega_D}$. Obviously, if the linearized operator $L_{u_D}$ admits only positive eigenvalues, then $u_D$ is nondegenerate. This is the case of stable solutions of \eqref{eq:pde_cone_starshaped}, which occur when $f$ satisfies conditions \textit{(i)} or \textit{(ii)} in Proposition \ref{prop:existence_radial_sector}, in particular, if $f$ is a constant. 

In general, $L_{u_D}$ could have negative eigenvalues, so to detect the nondegeneracy of $u_D$ we have to analyze the spectrum of the linear operator \eqref{eq:linearized_operator_radial_sector} in $H_0^1(\Omega_D \cup \Gamma_{1, 0})$. As we will see, the fact that $\Omega_D$ is a spherical sector in the cone $\Sigma_D$ (and not the ball $B_1$) plays a role.

The first remark is that zero is an eigenvalue for $L_{u_D}$ if and only if it is an eigenvalue for the following singular problem:
\begin{equation}
\label{eq:singular_eigenvalue_problem_radial_sector}
\left\{
\begin{array}{rcll}
- \Delta \psi - f'(u_D) \psi & = & \displaystyle \frac{\widehat \Lambda}{|x|^2} \psi & \quad \text{ in } \Omega_D \\[11pt]
\psi & = & 0 & \quad \text{ on } D \\[4pt]
\displaystyle \frac{\partial \psi}{\partial \nu} & = & 0 & \quad \text{ on } \Gamma_{1, 0} \setminus \{0\}.
\end{array}
\right.
\end{equation}

Therefore we investigate the eigenvalues of \eqref{eq:singular_eigenvalue_problem_radial_sector}. The advantage of considering this singular eigenvalue problem is that, since $u_D$ is radial, its eigenfunctions can be obtained by separation of variables, using polar coordinates in $\mathbb R^N$. To this aim we denote by $\{\lambda_j(D)\}_{j \in \mathbb N}$, the eigenvalues of the Laplace-Beltrami operator $- \Delta_{\mathbb{S}^{N - 1}}$ on the domain $D$ with Neumann boundary conditions. It is well-known that 
\begin{equation}
\label{eq:Neumann_eigenvalues_on_D}
0 = \lambda_0(D) < \lambda_1(D) \leq \lambda_2(D) \leq \ldots,
\end{equation}
and the only accumulation point is $+ \infty$. Then we consider the following singular eigenvalue problem in the interval $(0, 1)$:
\begin{equation}
\label{eq:singular_eigenvalue_problem_interval_cone}
\left\{
\begin{array}{rcll}
- z'' - \frac{N - 1}{r} z' - f'(u_D) z & = & \displaystyle \frac{\widehat \nu}{r^2} z & \quad \text{ in } (0, 1) \\[8pt]
z(1) & = & 0 &
\end{array}
\right.
\end{equation}
It is shown in \cite{AmadoriGladiali2020} (see also \cite{CiraoloPacellaPolvara2023}) that nonpositive eigenvalues for \eqref{eq:singular_eigenvalue_problem_interval_cone} can be defined. They are a finite number and we denote them by $\widehat \nu_i$, $i = 1, \ldots, k$. It is immediate to check that the eigenvalues $\widehat \nu_i$ are the eigenvalues of \eqref{eq:singular_eigenvalue_problem_radial_sector} which correspond to radial eigenfunctions. In particular, we consider the first eigenvalue $\widehat \nu_1$ of \eqref{eq:singular_eigenvalue_problem_interval_cone}, referring to \cite{AmadoriGladiali2020} for a variational definition and a study of its main properties.

By using \eqref{eq:singular_eigenvalue_problem_radial_sector}-\eqref{eq:singular_eigenvalue_problem_interval_cone} we obtain the following result:

\begin{prop}
\label{prop:sum_of_eigenvalues_cone}
The problem \eqref{eq:singular_eigenvalue_problem_radial_sector} admits zero as an eigenvalue if and only if there exist $i\in \mathbb{N}^+$ and $j\in\mathbb N$ such that
\begin{equation}
\label{eq:sum_eigenvalues_cone}
\widehat \nu_i + \lambda_j(D) = 0.
\end{equation}
\end{prop}

\begin{proof}
The proof follows by \cite[Proposition 2.6]{CiraoloPacellaPolvara2023}, where it is proved that the nonpositive eigenvalues of \eqref{eq:singular_eigenvalue_problem_radial_sector} are obtained by summing the eigenvalues of the one-dimensional problem \eqref{eq:singular_eigenvalue_problem_interval_cone} and the Neumann eigenvalues of $- \Delta_{\mathbb{S}^{N - 1}}$ on $D$. We refer also to \cite{DemarchisIanniPacella2017} for another approach, which consists in approximating the ball by annuli in order to avoid the singularity at $0$.
\end{proof}

From Proposition \ref{prop:sum_of_eigenvalues_cone} we get the following sufficient condition for a radial solution $u_D$ to be nondegenerate.

\begin{cor}
\label{cor:nondegeneracy_cone}
A radial solution $u_D$ of \eqref{eq:pde_cone_starshaped} in $\Omega_D$ (i.e. for $\varphi=0$) is nondegenerate if both the following conditions are satisfied:
\begin{enumerate}[(I)]
\item the eigenvalue problem \eqref{eq:singular_eigenvalue_problem_interval_cone} does not admit zero as an eigenvalue;

\item $\lambda_1(D) > - \widehat \nu_1$. 
\end{enumerate}
\end{cor}

\begin{proof}
From Condition \textit{(I)} we have
\begin{equation}\label{eq:condnui}
 \widehat \nu_i \neq 0 \quad \forall i\in \mathbb{N}^+,
 \end{equation}
which means that zero is not an eigenvalue of \eqref{eq:singular_eigenvalue_problem_radial_sector} with a corresponding radial eigenfunction. This, in turn, is equivalent to saying that zero is not a ``radial" eigenvalue of the linearized operator \eqref{eq:linearized_operator_radial_sector}, i.e., $u_D$ is a radial solution of \eqref{eq:pde_cone_starshaped} in $\Omega_D$ (or of \eqref{eq:Dirichlet_problem_in_the_ball} in $B_1$) which is nondegenerate in the subspace $H_{0, rad}^1(\Omega_D \cup \Gamma_{1, 0})$, which is the subspace of $H_{0}^1(\Omega_D \cup \Gamma_{1, 0})$  given by radial functions. %To conclude it remains to show that $u_D$ is non-degenerate in the whole $H_{0}^1(\Omega_D \cup \Gamma_{1, 0})$. %Therefore, $\widehat \nu_i \neq 0$ in formula \eqref{eq:sum_eigenvalues_cone}.

Now, since $\lambda_0(D) = 0$, $\lambda_1(D) > 0$ and since $\widehat \nu_1$ is the smallest eigenvalue of \eqref{eq:singular_eigenvalue_problem_interval_cone}, then, from Condition \textit{(II)} and \eqref{eq:condnui} we infer that the sum \eqref{eq:sum_eigenvalues_cone} can never be zero. Hence, thanks to Proposition \ref{prop:sum_of_eigenvalues_cone}, we have that zero is not an eigenvalue of \eqref{eq:singular_eigenvalue_problem_radial_sector} and so cannot be an eigenvalue for the linearized operator \eqref{eq:linearized_operator_radial_sector} in the whole $H_{0}^1(\Omega_D \cup \Gamma_{1, 0})$, i.e. $u_D$ is a non-degenerate solution to \eqref{eq:pde_cone_starshaped} in $\Omega_D$.
\end{proof}

\begin{remark}
Condition \textit{(I)} in Corollary \ref{cor:nondegeneracy_cone}, i.e., the nondegeneracy of $u_D$ in the space $H^1_{0, rad}(\Omega_D \cup \Gamma_{1, 0})$, is satisfied by positive radial solutions of \eqref{eq:pde_cone_starshaped} corresponding to many kinds of nonlinearities. 

It holds if $f$ satisfies conditions \textit{(i)} or \textit{(ii)} of Proposition \ref{prop:existence_radial_sector}, because in this case all the eigenvalues of \eqref{eq:linearized_operator_radial_sector} and of \eqref{eq:singular_eigenvalue_problem_radial_sector} are positive. It then follows that \textit{(II)} holds as well. More precisely, in the case \textit{(i)}, since $0<a < \mu_1(B_1)$, the first eigenvalue of $L_{u_D}$ is positive, so
\begin{equation}
\underbrace{\lambda_0(D)}_{= 0} + \widehat \nu_1 > 0.
\end{equation}

In the case \textit{(ii)}, since $f'(u_D) \leq 0$, it follows that $\widehat \nu_1 > 0$.

Among the nonlinearities satisfying condition \textit{(iii)} of Proposition \ref{prop:existence_radial_sector} we could consider $f(u) = u^p$, $1 < p < \frac{N + 2}{N - 2}$, $N \geq 3$. Then it is known that the positive radial solution of \eqref{eq:Dirichlet_problem_in_the_ball} is unique and nondegenerate (see \cite{DamascelliGrossiPacella1999, GidasNiNirenberg1979}), so \textit{(I)} holds. It is also well-known that for this nonlinearity it holds $\widehat \nu_1 < 0$ and $\widehat \nu_1$ is the only negative eigenvalue of \eqref{eq:singular_eigenvalue_problem_interval_cone}, because $u_D$ can be obtained by the Mountain Pass Theorem or by minimization on the Nehari manifold and thus it has Morse index one. Then the validity of \textit{(II)} depends on the cone, since it depends on $\lambda_1(D)$. However, once $p$ is fixed, since $\widehat \nu_1$ does not depend on the cone, it is obvious that, by varying $D$, there are many cones for which \textit{(II)} holds. Moreover, it has been proved in \cite{CiraoloPacellaPolvara2023} that $\widehat \nu_1 > - (N - 1)$ for every autonomous nonlinearity, so that whenever $\lambda_1(D) > N - 1$ all radial solutions of \eqref{eq:pde_cone_starshaped} are nondegenerate.
\end{remark}

\subsection{Stability of $\mathbf{(\Omega_{\textit{D}}, \textit{u}_{\textit{D}})}$}
\label{subsec:stability_cone}

Let us first observe that if $u_D$ is a positive nondegenerate radial solution of \eqref{eq:pde_cone_starshaped} for $\varphi = 0$, belonging to $W^{1, \infty}(\Omega_D) \cap W^{2, 2}(\Omega_D)$, then $(\Omega_D, u_D)$ is energy-stationary in the sense of Definition \ref{def:energy_stationary}. Indeed, since $u_D$ is radial, we have that $\frac{\partial u_D}{\partial \nu} = \text{constant}$ on $\Gamma_0 = D$ and thanks to Proposition \ref{prop:energy_stationary_char} we easily conclude.

To investigate the stability of $(\Omega_D, u_D)$ we analyze the quadratic form corresponding to the second derivative $I''(\varphi)$ at $\varphi = 0$. Fixing the constant $c$ in the definition of $M$ (see \eqref{eq:def_M}) as $c = |\Omega_D|$, we have that the tangent space to $M$ at $\varphi = 0$ is given by
\begin{equation}
\label{eq:T_0_M}
T_0 M = \left\{v \in C^2(\overline D, \mathbb R) \ : \ \int_D v \ d\sigma = 0 \right\}.
\end{equation}

Writing $u_D(r) = u_D(|x|)$, we denote by $u_D'$ and $u_D''$ the derivatives of $u_D$ with respect to $r$, so that
\begin{equation}
\label{eq:boundary_values_u_D_primes}
u_D'(1) = \left.\frac{\partial u_D}{\partial \nu} \right|_D, \qquad u_D''(1) = [(D^2u_D \nu) \cdot \nu] |_D.
\end{equation}
By Hopf's Lemma we know that $u_D'(1) < 0$ and actually
\begin{equation}
\label{eq:mu_zero}
u_D'(1) = - \sqrt{- 2 \mu_D},
\end{equation} 
where $\mu_D$ denotes the Lagrange multiplier in the case $\varphi = 0$, see \eqref{eq:Lagrange_multiplier}. 

For $v \in T_0 M$, we will denote by $\widetilde u_v$ the solution of
\begin{equation}
\label{eq:equation_u_tilde_v}
\left\{
\begin{array}{rcll}
- \Delta \widetilde u_v - f'(u_D) \widetilde u_v & = & 0 & \quad \text{ in } \Omega_D \\[5pt]
\widetilde u_v & = & - u_D'(1) v & \quad \text{ on } D \\[4pt]
\displaystyle \frac{\partial \widetilde u_v}{\partial \nu} & = & 0 & \quad \text{ on } \Gamma_{1, 0} \setminus \{0\}
\end{array}
\right.
.
\end{equation}

Let us remark that for every $q \in D$ the outer unit normal vector $\nu(q)$ is precisely $q$, hence \eqref{eq:equation_u_tilde_v} corresponds to \eqref{eq:equation_for_u_tilde} in $\Omega_D$.

Note that, since $u_D$ is a nondegenerate radial solution, then the weak solution $\widetilde u_v$ of \eqref{eq:equation_u_tilde_v} is unique for every $v$.

Our next result shows that the quadratic form corresponding to the second derivative of $I$ at $\varphi = 0$ has a simple expression.

\begin{lemma}
\label{lem:I_two_primes_zero}
For any $v \in T_0M$ it holds
\begin{equation}
\label{eq:I_two_primes_zero}
I''(0)[v, v] = - u_D'(1)\left(\int_D v \frac{\partial \widetilde u_v}{\partial \nu} \ d \sigma + u_D''(1) \int_D v^2 \ d\sigma \right),
\end{equation}
where $\widetilde u_v$ is the solution of \eqref{eq:equation_u_tilde_v}.
\end{lemma}

\begin{proof}
From Lemma \ref{lem:T_2_cone}, \eqref{eq:volume_derivatives} and Lemma \ref{lem:Lagrange_multipliers_second_derivatives}, with $w = v$, by simple substitutions and elementary computations we obtain:
\begin{align}
I''(0)[v, v]
& = - \frac{N}{2} \int_D (u_D'(1))^2 v^2 \ d\sigma - \int_D u_D'(1) v \frac{\partial \widetilde u_v}{\partial \nu} \ d\sigma \nonumber \\
& \quad - \int_D u_D'(1) v^2 (D^2u_D \nu)\cdot \nu \ d\sigma - N \mu_D \int_D v^2 \ d\sigma. \label{eq:mani_main}
\end{align}

Since $\widetilde u_v = -u_D'(1) v$ on $D$, by \eqref{eq:boundary_values_u_D_primes} and \eqref{eq:mu_zero}, we deduce that
\begin{align}
& - \frac{N}{2} \int_D (u_D'(1))^2 v^2 \ d\sigma = - \frac{N}{2} \int_D \widetilde u_v^2 \ d\sigma; \label{eq:mani_1} \\
& - N \mu_D \int_D v^2 \ d\sigma = \frac{N}{2} \int_D \widetilde u_v^2 \ d\sigma. \label{eq:mani_3}
\end{align}
Then \eqref{eq:I_two_primes_zero} follows by substituting \eqref{eq:mani_1}-\eqref{eq:mani_3} into \eqref{eq:mani_main}.
\end{proof}

To investigate the stability of $(\Omega_D, u_D)$ as an energy stationary pair for $I$ we need to study the solution $\widetilde u_v$ of \eqref{eq:equation_u_tilde_v}, for any $v \in T_0M$ (that is, for functions with mean value zero on $D$). As we will see, it will be enough to consider only functions $v$ which are eigenfunctions of the Laplace-Beltrami operator $- \Delta_{\mathbb{S}^{N - 1}}$ with Neumann boundary conditions on $D$. Hence we consider the eigenvalue problem
\begin{equation}
\label{eq:Neumann_eigenvalue_problem_D}
\left\{ 
\begin{array}{rcll}
- \Delta \psi & = & \lambda \psi & \quad \text{ on } D \\[3pt]
\displaystyle \frac{\partial \psi}{\partial \nu} & = & 0 & \quad \text{ on } \partial D
\end{array}
\right.
\end{equation}
and denote its eigenvalues as in \eqref{eq:Neumann_eigenvalues_on_D}, counted with multiplicity: $0 = \lambda_0(D) < \lambda_1(D)\leq \lambda_2(D) \leq \ldots$. The corresponding $L^2$-normalized eigenfunctions are denoted by $\{\psi_j\}_{j \in \mathbb N}$, with $\int_D \psi_j^2 \ d \sigma = 1$, $\psi_0 = \text{constant}$ and $\int_D \psi_j \ d\sigma = 0$ for $j \geq 1$.

\begin{theorem}
Let $j \geq 1$ and $\widetilde u_j$ be the unique solution of \eqref{eq:equation_u_tilde_v} for $v = \psi_j$. Then, writing $\widetilde u_j=\widetilde u_j(r,q)$, the function
\begin{equation}
\label{eq:def_h}
h_j(r) = \int_D \widetilde u_j(r, q) \psi_j(q) \ d\sigma, \quad r \in (0, 1)
\end{equation}
satisfies
\begin{equation}
\label{eq:equation_for_h_cone}
\left\{
\begin{array}{ll}
\displaystyle - h_j'' - \frac{N - 1}{r} h_j' - f'(u_D) h_j = - \frac{\lambda_j(D)}{r^2} h_j & \quad \text{ in } (0, 1) \\[8pt]
h_j(1) = -u_D'(1)
\end{array}
\right.
\end{equation}
\end{theorem}

\begin{proof}
Since the proof is the same for all $j$, we drop the index and the dependence on $D$ and write simply $h$, $\psi$ and $\lambda$.

It is immediate to check that $h(1) = - u_D'(1)$. Moreover, since we can bring the radial derivative inside the integral on $D$, for every $r \in (0, 1]$ we have:
\begin{align}
- h''(r) - \frac{N - 1}{r} h'(r)
& = \int_D \left(- \widetilde u_{rr}(r, q) - \frac{N - 1}{r} \widetilde u_r(r, q) \right) \psi(q) \ d \sigma \nonumber \\
& = \int_D \left(- \Delta \widetilde u + \frac{1}{r^2} \Delta_{\mathbb{S}^{N - 1}} \widetilde u \right) \psi \ d \sigma \nonumber \\
& = \int_D f'(u_D(r)) \widetilde u \psi \ d\sigma + \frac{1}{r^2} \int_D (\Delta_{\mathbb{S}^{N - 1}} \widetilde u) \psi \ d\sigma. \label{eq:radial_laplacian_of_h}
\end{align}
Now, on the one hand, 
\begin{equation}
\label{eq:linearized_term_for_h}
\int_D f'(u_D(r)) \widetilde u \psi \ d\sigma = f'(u_D(r)) h(r).
\end{equation}
On the other hand, applying Green's formula, taking into account the Neumann conditions on $\psi$ and $\widetilde u$, we infer that
\begin{align}
\frac{1}{r^2} \int_D (\Delta_{\mathbb{S}^{N - 1}} \widetilde u) \psi \ d \sigma
& = \frac{1}{r^2} \int_D \widetilde u \Delta_{\mathbb{S}^{N - 1}} \psi \ d\sigma \nonumber \\
& = - \frac{\lambda}{r^2} \int_D \widetilde u \psi \ d\sigma \nonumber \\
& = - \frac{\lambda}{r^2} h(r). \label{eq:int_by_parts_to_find_h}
\end{align}

Substituting \eqref{eq:linearized_term_for_h} and \eqref{eq:int_by_parts_to_find_h} into \eqref{eq:radial_laplacian_of_h} we conclude the proof.
\end{proof}

\begin{remark}
\label{rem:u_equals_h_psi_cone}
Note that with $\widetilde u_j$ and $h_j$ as in Theorem \ref{thm:exists_h} we have that
\begin{equation*}
\widetilde u_j(r, q) = h_j(r) \psi_j(q).
\end{equation*}
Indeed, the boundary conditions are clearly satisfied by this function, and it holds
\begin{align*}
- \Delta (h_j \psi_j)
& = - h_j'' \psi_j - \frac{(N - 1)}{r} h_j' \psi_j - \frac{h_j}{r^2} \Delta_{\mathbb{S}^{N - 1}} \psi_j\\[-2pt]
& = f'(u_D) h_j \psi_j - \frac{\lambda_j(D)}{r^2} h_j \psi_j + \frac{\lambda_j(D)}{r^2} h_j \psi_j \\
& = f'(u_D) h_j \psi_j.
\end{align*}
\end{remark}

\begin{prop}
\label{prop:h(0)_equals_0}
Let $N \geq 3$. For any $j \geq 1$ we have
\begin{equation}
\label{eq:DGG_A_30}
\int_0^1 r^{N - 3} h_j^2 \ dr < + \infty
\end{equation}
and 
\begin{equation}
\label{eq:DGG_A_29.3}
\int_0^1 r^{N - 1} (h_j')^2 \ dr < + \infty.
\end{equation}
Moreover, $h_j \in L^\infty(0, \infty)$ and $h_j(0) = 0$.
\end{prop}

\begin{proof}
Again, for simplicity, we drop the index $j$. Since $\widetilde u \in H^1(\Omega_D)$ (see Sect. \ref{sec:general_unbounded_domains}), writing $\widetilde u=\widetilde u(r,q)$ and recalling that $\psi$ is a $L^2(D)$-normalized solution to \eqref{eq:Neumann_eigenvalue_problem_D}, we get that
\begin{align}
+ \infty
& > \int_{\Omega_D}|\nabla \widetilde u|^2 \ dx \nonumber \\
& = \int_0^1 r^{N - 1} (h')^2 \int_D \psi^2 \ d\sigma \ dr + \int_0^1 r^{N - 3} h^2 \int_D |\nabla_{\mathbb{S}^{N - 1}} \psi|^2 \ d\sigma \ dr \nonumber \\
& = \int_0^1 r^{N - 1}(h')^2 \ dr + \lambda \int_0^1 r^{N - 3} h^2 \ dr,\nonumber
\end{align}
which proves \eqref{eq:DGG_A_30} and \eqref{eq:DGG_A_29.3}. Once we have these estimates, we can proceed as in \cite[Lemma A.9]{DancerGladialiGrossi2017} to get the boundness of $h$ and $h(0) = 0$.
\end{proof}

\begin{prop}
\label{prop:h_positive_cone}
Let $\lambda_j(D)$, $j \geq 1$ be a nontrivial Neumann eigenvalue of $- \Delta_{\mathbb{S}^{N - 1}}$ on $D$. Assume that 
\begin{equation*}
- \widehat \nu_1 < \lambda_j(D),
\end{equation*}
where $\widehat \nu_1$ is the smallest eigenvalue of \eqref{eq:singular_eigenvalue_problem_interval_cone}. Then for the solution $h_j$ of \eqref{eq:equation_for_h_cone} it holds that
\begin{equation*}
h_j > 0 \quad \text{ in } \quad (0, 1).
\end{equation*}
\end{prop}

\begin{proof}
Let $z_1$ be an $L^2$-normalized first eigenfunction of \eqref{eq:singular_eigenvalue_problem_interval_cone}. From \cite[Section 3.1]{AmadoriGladiali2020} we know that $z_1$ does not change sign. 

Writing the equations satisfied by $h_j$ and $z_1$ in Sturm-Liouville form we have:
\begin{align*}
& (r^{N - 1} h_j')' + r^{N - 1}(f'(u_D) - r^{- 2} \lambda_j(D))h_j = 0, \\
& (r^{N - 1} z_1')' + r^{N - 1}(f'(u_D) + r^{-2} \widehat \nu_1)z_1 = 0.
\end{align*}
By Proposition \ref{prop:h(0)_equals_0} we know that $h_j(0) = 0$ and $h_j(1) = -u_D'(1)>0$. 

Now, assume by contradiction that $h_j$ changes sign in $(0, 1)$. Then there would exist $r_0 \in (0, 1)$ such that $h_j(0) = 0$. Since $- \widehat \nu_1 < \lambda_j(D)$, then, by the Sturm-Picone Comparison Theorem it would follow that $z_1$ has a zero in $(0, r_0)$. This is a contradiction, because $z_1$ does not change sign. Hence the only possibility is that $h_j$ is strictly positive in $(0, 1)$.
\end{proof}

We are ready to prove our main result for problem \eqref{eq:pde} in the case of the cone, i.e., Theorem \ref{thm:stability_cone}, which is a sharp instability/stability result for the pair $(\Omega_D, u_D)$.

\subsection*{Proof of Theorem \ref{thm:stability_cone}}
Let us fix the domain $D$ which spans the cone, so that we denote $\lambda_1(D)$ simply by $\lambda_1$.

For \textit{(i)}, let $\widetilde u_1 = h_1 \psi_1$ be the solution of \eqref{eq:equation_u_tilde_v} with $v = \psi_1$. Then
\begin{equation}
\label{eq:I''_psi_one_psi_one}
I''(0)[\psi_1, \psi_1] = - u_D'(1) (h_1'(1) + u_D''(1)).
\end{equation}

Putting \eqref{eq:equation_for_h_cone} in Sturm-Liouville form we get
\begin{equation}
\label{eq:sturm_h}
- (r^{N - 1} h_1')' - r^{N - 1}f'(u_D)h_1 = - r^{N - 3} \lambda_1 h_1.
\end{equation}
On the other hand, writing $- \Delta u_D = f(u_D)$ in polar coordinates and differentiating with respect to $r = |x|$ we  get
\begin{equation*}
- (u_D')'' - \frac{N - 1}{r} (u_D')' - f'(u_D) u_D' = - \frac{N - 1}{r^2} u_D',
\end{equation*}
which in Sturm-Liouville form is
\begin{equation}
\label{eq:sturm_u_D'}
- (r^{N - 1} u_D'')' - r^{N - 1} f'(u_D)u_D' = - r^{N - 3} (N - 1) u_D'.
\end{equation}

Multiplying \eqref{eq:sturm_h} by $u_D'$ and integrating by parts in $(\bar r, 1)$ we get that
\begin{equation}\label{eq:proofteo1}
\begin{array}{lll}
 %&&\displaystyle- \int_{\bar r}^1 (r^{N - 1} h_1')' u_D' \ dr - \int_{\bar r}^1 r^{N - 1} f'(u_D) h_1 u_D' \ dr \\[12pt]
&&\displaystyle\int_{\bar r}^1 r^{N - 1} h_1' u_D'' \ dr - (r^{N - 1} h_1' u_D')\big|_{\bar r}^1- \int_{\bar r}^1 r^{N - 1} f'(u_D) h_1 u_D' \ dr\\[6pt]
&=&\displaystyle - \lambda_1 \int_{\bar r}^1 r^{N - 3} h_1 u_D' \ dr.
\end{array}
\end{equation}
 Similarly, multiplying \eqref{eq:sturm_u_D'} by $h_1$ and integrating by parts we deduce that
\begin{equation}\label{eq2:proofteo1}
\begin{array}{lll}
&\displaystyle\int_{\bar r}^1 r^{N - 1} h_1' u_D'' \, dr - (r^{N - 1} h_1 u_D'')\big|_{\bar r}^1 - \int_{\bar r}^1 r^{N - 1} f'(u_D) u_D' h_1 \, dr\\
=&\displaystyle -(N - 1) \int_{\bar r}^1 r^{N - 3} u_D' h_1 \, dr.
\end{array}
\end{equation}
Notice that, in view of Proposition \ref{prop:h(0)_equals_0}, the right-hand sides of \eqref{eq:proofteo1}, \eqref{eq2:proofteo1} remain finite when taking the limit as $\bar r\to 0^+$. In addition, we claim that
\begin{equation}
\label{eq:claimteo1}
\lim_{\bar r \to 0^+} r^{N - 1} h_1'(\bar r) u_D'(\bar r) = 0.
\end{equation}
Indeed, integrating \eqref{eq:sturm_h} and taking the absolute value we obtain
\begin{align*}
\left|\int_{\bar r}^1 - (r^{N - 1} h_1')' \ dr \right|
& = \left|\bar r^{N - 1} h_1'(\bar r) - h_1'(1)\right| \\
& \leq \int_{\bar r} r^{N - 1} |f'(u_D)| h_1\, dr + \int_0^1 r^{N - 3} \lambda_1 h_1 \, dr \\
& \leq C_1
\end{align*}
for some $C_1 > 0$. Hence
\begin{equation}
\limsup_{\bar r \to 0^+} \bar r^{N - 1}|h_1'(\bar r)| \leq C_2
\end{equation}
for some $C_2>0$, and thus, since $\lim_{\bar r \to 0^+} u_D'(\bar r) = 0$, \eqref{eq:claimteo1} follows.

Now, subtracting \eqref{eq2:proofteo1} from \eqref{eq:proofteo1} and taking the limit as $\bar r \to 0^+$, then, thanks to \eqref{eq:claimteo1} and since $h_1(0)=0$, $h_1(1) = - u_D'(1)$, we obtain
\begin{equation}
\label{eq:equation_for_instability}
- u_D'(1)(h_1'(1) + u_D''(1)) = (N - 1 - \lambda_1) \int_0^1 r^{N - 3} h_1 u_D' \ dr.
\end{equation}
Since $\lambda_1 > -\widehat \nu_1$, then, by Proposition \ref{prop:h_positive_cone}, we have that $h_1 > 0$ in $(0,1)$. On the other hand $u_D' < 0$ in $(0, 1)$ and $\lambda_1 < N - 1$ by assumption. Hence by \eqref{eq:I''_psi_one_psi_one} and \eqref{eq:equation_for_instability} we obtain $$I''(0)[\psi_1, \psi_1] < 0,$$ 
which proves \textit{(i)}.\\

For \textit{(ii)}, we choose an orthonormal basis $(\psi_j)_j$ of $L^2(D)$ made of normalized eigenfunctions of \eqref{eq:Neumann_eigenvalue_problem_D}. Then any $v \in T_0M$ can be written as
\begin{equation*}
v = \sum_{j = 1}^\infty (v, \psi_j) \psi_j,
\end{equation*}
where $(\cdot, \cdot)$ denotes the inner product in $L^2(D)$. We assume without loss of generality that $\int_D v^2 \ d\sigma = 1$. Let $\widetilde u_j$ be the solution of \eqref{eq:equation_u_tilde_v} with $v = \psi_j$, then we can check that
\begin{equation*}
\widetilde v = \sum_{j = 1}^\infty (v, \psi_j) \widetilde u_j
\end{equation*}
is the solution of \eqref{eq:equation_u_tilde_v}. As observed in Remark \ref{rem:u_equals_h_psi_cone}, $\widetilde u_j(r,q) = h_j(r) \psi_j(q)$ for every $j \in \mathbb N$, so
\begin{equation*}
\frac{\partial \widetilde u_j}{\partial \nu}(1,q) = h_j'(1) \psi_j(q)\ \quad\text{on $D$. }
\end{equation*}

By an argument analogous to the one presented in the proof of \textit{(i)}, we have that if $k > j$, then $h_k'(1) {\geq} h_j'(1)$ and in fact $h_k'(1) > h_j'(1)$ if $k>j$ are such that $\lambda_k > \lambda_j$. 

Indeed, writing the equations for $h_j, h_k$, multiplying the first one by $h_k$ and the second one by $h_j$, integrating by parts and subtracting we get
\begin{equation*}
- u_D'(1) (h_k'(1) - h_j'(1)) = (- \lambda_j + \lambda_k) \int_0^1 r^{N - 3}h_ih_j \geq 0.
\end{equation*}

Exploiting the orthogonality of the basis $(\psi_j)_j$ and exploiting \eqref{eq:equation_for_instability} we obtain
\begin{align}
I''(0)[v, v]
& = - u_D'(1) \left(\int_D \left(\sum_{j = 1}^\infty (v, \psi_j) \psi_j \right)\left(\sum_{k = 1}^\infty (v, \psi_j) h_k'(1) \psi_k\right) \ d\sigma + u_D''(1) \int_D v^2 \ d\sigma \right) \nonumber \\
& = - u_D'(1) \left(\left(\sum_{j = 1}^\infty (v, \psi_j)^2 h_j'(1) \right) + u_D''(1)\right) \nonumber \\
& \geq - u_D'(1) \left(h_1'(1) \left(\sum_{j = 1}^\infty (v, \psi_j)^2 \right) + u_D''(1) \right) \nonumber \\
& = - u_D'(1) (h_1'(1) + u_D''(1)) \nonumber\\ 
& = \nonumber (N - 1 - \lambda_1) \int_0^1 r^{N - 3} h_1 u_D' \ dr>0,
\end{align}
because $h_1>0$ in $(0,1)$, $u^\prime_D<0$ in $(0,1)$ and $\lambda_1>N-1$ by assumption. The proof is complete.
\qed

\begin{remark}
\label{rem:remark_RN_stability}
As already pointed out in Remark \ref{rem:remark_RN_energy_stationary}, in the case when $\mathcal C = \mathbb R^N$, the couples $(B, u_B)$, where $B$ is a ball and $u_B$ is a positive nondegenerate radial solution, are the only energy-stationary pairs. Thus it remains to study the stability of $(B, u_B)$ as critical point of the energy functional $T$. This can be done by looking at the problem as the case of a cone spanned by the domain $D = \mathbb{S}^{N - 1}$. 

As observed in Remark \ref{rem:Remark_on_thm_stability_cone}, the first eigenvalue $\widehat \nu_1$ of the singular eigenvalue problem \eqref{eq:singular_eigenvalue_problem_interval_cone} is always larger than $-(N - 1)$. On the other hand, it is known that the first nontrivial eigenvalue of the Laplace-Beltrami operator on the whole $\mathbb{S}^{N - 1}$ is precisely $N - 1$. Then any radial solution $u_B$ is nondegenerate and we obtain that the pair $(B,  u_B)$ is a semistable stationary-point.
\end{remark}

\section{The case of the cylinder}
\label{sec:cylinder}

Let $\omega \subset \mathbb R^{N - 1}$ be a smooth bounded domain and let $\Sigma_\omega$ be the half-cylinder spanned by $\omega$, namely
\begin{equation*}
\Sigma_\omega \coloneqq \omega \times (0, + \infty).
\end{equation*}
We denote by $x = (x', x_N)$ the points in $\overline \Sigma_\omega$, where $x'=(x_1,\ldots,x_{N-1}) \in \overline \omega$ and $x_N \geq 0$.

In analogy with the case of the cone, we consider domains whose relative boundaries are the cartesian graphs of functions in $C^2(\overline \omega)$. More precisely, for $\varphi \in C^2(\overline \omega)$ we set
\begin{equation*}
\Gamma_\varphi \coloneqq \{(x', x_N) \in \Sigma_\omega \ : \ x_N = e^{\varphi(x')}\}
\end{equation*}
and consider domains of the type
\begin{equation*}
\Omega_\varphi = \{(x', x_N) \in \Sigma_\omega \ : \ x_N < e^{\varphi(x')}\}.
\end{equation*}
Finally, let
\begin{equation*}
\Gamma_{1, \varphi} \coloneqq (\partial \Omega_\varphi\setminus\overline{\Gamma}_\varphi).
\end{equation*}

Observe that the outer unit normal vector on $\Gamma_\varphi$ at a point $(x', e^{\varphi(x')})$ is given by
\begin{equation}
\label{eq:nu}
\nu=\nu_\varphi(x^\prime)= \frac{(- e^{\varphi(x^\prime)}  \nabla_{\R^{N-1}}\varphi(x^\prime), 1)}{\sqrt{1 + |e^{\varphi(x^\prime)} \nabla_{\R^{N-1}} \varphi(x^\prime)|^2}},
\end{equation}
where $\nabla_{\R^{N-1}}$ denotes the gradient with respect to the variables $x_1,\ldots,x_{N-1}$.

\subsection{Energy functional in cylindrical domains}
\label{subsec:energy_cylinder}

We study the semilinear elliptic problem
\begin{equation}
\label{eq:pde_cylinder}
\left\{
\begin{array}{rcll}
- \Delta u & = & f(u) & \quad \text{ in } \Omega_\varphi \\[2mm]
u & = & 0 & \quad \text{ on } \Gamma_\varphi \\ [1mm]
\displaystyle \frac{\partial u}{\partial \nu} & = & 0 & \quad \text{ on } \Gamma_{1, \varphi}
\end{array}
\right.
\end{equation}
and consider bounded positive weak solutions of \eqref{eq:pde_cylinder} in the Sobolev space $H_0^1(\Omega_\varphi \cup \Gamma_{1,\varphi})$, which is the space of functions in $H^1(\Omega_\varphi)$ whose trace vanishes on $\Gamma_\varphi$. 

As before, we assume that a bounded nondegenerate positive solution $u_\varphi$ of \eqref{eq:pde_cylinder} exists and belongs to $W^{1, \infty}(\Omega_\varphi) \cap W^{2, 2}(\Omega_\varphi)$, so that we can apply the results of Section \ref{sec:general_unbounded_domains}.

We consider variations of the domain $\Omega_\varphi$ in the class of cartesian graphs of the type $\Omega_{\varphi + tv}$, for $v \in C^2(\overline \omega)$, which amounts to consider a one-parameter family of diffeomorphisms $\xi:(-\eta,\eta)\times\overline\Sigma_\omega\to \overline\Sigma_\omega$ of the type
\begin{equation*}
\xi(t, x) = (x', e^{t v(x')}x_N),
\end{equation*}
whose inverse, for any fixed $t\in(-\eta, \eta)$, is given by 
\begin{equation*}
\xi(t, x)^{-1} = (x', e^{- t v(x')} x_N) = \xi(- t, x).
\end{equation*}
This one-parameter family of diffeomorphisms is generated by the vector field
\begin{equation}\label{eq:defvectfield}
V(x) = (0^\prime, v(x')x_N),
\end{equation}
where $0^\prime:=(0, \ldots, 0)\in\R^{N-1}$.
Indeed, $\xi(0, x) = x$ for every $x \in \overline\Sigma_\omega$, 
\begin{equation*}
\frac{d\xi}{dt}(t, x) = (0^\prime, e^{tv(x')}v(x')x_N) = V(\xi(t, x))\quad \forall (t,x)\in (-\eta,\eta)\times \Sigma_\omega
\end{equation*}
and $\xi(t, x)\in\partial\Sigma_\omega$, for all $(t,x)\in (-\eta,\eta)\times \partial\Sigma_\omega$.
We also observe that, in view of \eqref{eq:nu}, it holds
\begin{equation}
\label{eq:V_scalar_nu}
\langle V, \nu \rangle = \left\langle (0^\prime, v e^\varphi), \frac{(- e^\varphi \nabla_{\R^{N-1}} \varphi, 1)}{\sqrt{1 + |e^\varphi \nabla_{\R^{N-1}} \varphi|^2}} \right\rangle =  \frac{v e^\varphi}{\sqrt{1 + |e^\varphi \nabla_{\R^{N-1}} \varphi|^2}} \quad \ \text{on $\Gamma_\varphi$.}
\end{equation}

The energy functional $T$ defined in \eqref{eq:def_T} becomes a functional depending only on functions in $C^2(\overline \omega)$. More precisely, for every $v \in C^2(\overline \omega)$, in view of Proposition \ref{prop:local_uniqueness}, there exists $\delta > 0$ sufficiently small such that for all $t \in (- \delta, \delta)$
\begin{equation*}
T(\varphi + tv) = T(\Omega_{\varphi + tv}) = J(u_{\varphi + tv}),
\end{equation*}
is well defined, where $u_{\varphi + tv} \coloneqq u_{\Omega_{\varphi + tv}}$ is the unique positive solution of \eqref{eq:pde_cylinder} in the domain $\Omega_{\varphi + tv}$, in a neighborhood of $u_\varphi \circ \xi_t^{-1}$.

By the results of Section \ref{sec:general_unbounded_domains} we know that the map $t \mapsto u_{\varphi + tv}$ is differentiable at $t = 0$, and the derivative $\widetilde u$ is a weak solution of
\begin{equation}
\label{eq:equation_for_u_tilde_cartesian_graphs}
\left\{
\begin{array}{rcll}
- \Delta \widetilde u & = & f'(u_\varphi) \widetilde u & \quad \text{ in } \Omega_\varphi \\ [.6em]
\widetilde u & = & \displaystyle - \frac{\partial u_\varphi}{\partial \nu} \frac{v e^\varphi}{\sqrt{1 + |e^\varphi \nabla_{\R^{N-1}} \varphi|^2}} & \quad \text{ on } \Gamma_\varphi \\
\displaystyle \frac{\partial \widetilde u}{\partial \nu} & = & 0 & \quad \text{ on } \Gamma_{1, \varphi}
\end{array}
\right.
\end{equation}

We now compute the first derivative of $T$ at $\Omega_\varphi$, i.e., for $t = 0$, with respect to variations $v \in C^2(\overline \omega)$.

\hfill

\begin{lemma}
\label{lem:T_1_cylinder}
Let $\varphi \in C^2(\overline \omega)$ and assume that $u_\varphi$ is a positive nondegenerate solution of \eqref{eq:pde_cylinder} which belongs to $W^{1, \infty}(\Omega) \cap W^{2, 2}(\Omega)$. Then, for any $v \in C^2(\overline \omega)$ we have
\begin{equation}
\label{eq:T_1_cylinder}
T'(\varphi)[v] = - \frac{1}{2} \int_\omega \left(\frac{\partial u_\varphi}{\partial \nu}(x^\prime,e^\varphi)\right)^2 v e^\varphi \ dx'.
\end{equation}
\end{lemma}

\begin{proof}
The proof is similar to that of Lemma \ref{lem:T1_radial_graphs}. It suffices to observe that for the parametrization of $\Gamma_\varphi$ given by $x=(x^\prime,e^{\varphi(x^\prime)})$, for $x^\prime\in\omega$, the induced $(N-1)$-dimensional area element on $\Gamma_\varphi$ is expressed by
\begin{equation*}
d \sigma_{\Gamma_\varphi} = \sqrt{1 + |e^\varphi \nabla_{\R^{N-1}} \varphi|^2} \ dx'.
\end{equation*}
Then the result follows immediately from Proposition \ref{prop:T_1_general}, taking into account  \eqref{eq:V_scalar_nu}. 
\end{proof}

\hfill

\begin{lemma}
\label{lem:T2_cartesian}
Let $\varphi$ and $u_\varphi$ be as in Lemma \ref{lem:T_1_cylinder}. Then for any $v, w \in C^2(\overline \omega)$ it holds
\begin{align}
T''(\varphi)[v, w] =
& - \frac{1}{2} \int_\omega \left(\frac{\partial u_\varphi}{\partial \nu}(x^\prime,e^\varphi) \right)^2 e^\varphi v w \ dx' \nonumber \\
& - \int_\omega \frac{\partial \widetilde u_w}{\partial \nu}(x^\prime,e^\varphi) \frac{\partial u_\varphi}{\partial \nu}(x^\prime,e^\varphi) e^\varphi v \ dx' \nonumber \\
& - \int_\omega \frac{\partial u_\varphi}{\partial \nu}(x^\prime,e^\varphi) [(D^2u_\varphi(x^\prime,e^\varphi) (0^\prime,e^\varphi)) \cdot \nu]  vw \ dx' \nonumber \\
& + \int_\omega \frac{\partial u_\varphi}{\partial \nu}(x^\prime,e^\varphi)e^{2\varphi} v  \frac{\nabla u_\varphi(x^\prime,e^\varphi) \cdot (w \nabla_{\mathbb R^{N - 1}} \varphi + \nabla_{\mathbb R^{N - 1}} w, 0)}{\sqrt{1 + |e^\varphi \nabla_{\mathbb R^{N - 1}} \varphi|^2}} \ dx' \nonumber \\
& + \int_\omega \left(\frac{\partial u_\varphi}{\partial \nu}(x^\prime,e^\varphi) \right)^2 e^{3\varphi} v \frac{\nabla_{\mathbb R^{N - 1}} \varphi \cdot (w\nabla_{\mathbb R^{N - 1}} \varphi + \nabla_{\mathbb R^{N - 1}} w)}{1 + |e^\varphi \nabla_{\mathbb R^{N - 1}} \varphi|^2} \ dx', 
\label{eq:T2_cartesian}
\end{align}
where  $\widetilde u_w$ is the solution of \eqref{eq:equation_for_u_tilde_cartesian_graphs}, with $w$ in the place of $v$.
\end{lemma}

\begin{proof}
Let $v, w \in C^2(\overline \omega)$. By definition, Lemma \ref{lem:T_1_cylinder} and using the Leibniz rule, we have:
\begin{align}
T''(\varphi)[v, w] 
& = \left. \frac{d}{ds} \right|_{s = 0}\left(- \frac{1}{2} \int_\omega \left(\frac{\partial u_{\varphi + sw}}{\partial \nu}(x^\prime,e^{\varphi + sw})\right)^2 e^{\varphi + sw} v \ dx' \right) \nonumber \\
& = - \int_\omega e^\varphi v \frac{\partial u_\varphi}{\partial \nu} \left. \frac{d}{ds}\right|_{s = 0}\left(\frac{\partial u_{\varphi + sw}}{\partial \nu}(x^\prime,e^{\varphi + sw}) \right)  \ dx' \nonumber \\
&\ \ \ - \frac{1}{2} \int_\omega \left(\frac{\partial u_\varphi}{\partial \nu}(x^\prime,e^{\varphi}) \right)^2 e^\varphi v w \ dx'. \label{eq:second_variation}
\end{align}
To conclude it suffices to compute the derivative in the first integral of the right-hand side of \eqref{eq:second_variation}. To this end we observe that
\begin{align}
\left. \frac{d}{ds}\right|_{s = 0} \left(\frac{\partial u_{\varphi + sw}}{\partial \nu}(x^\prime,e^{\varphi + sw}) \right) 
& = \left. \frac{d}{ds}\right|_{s = 0} \left(\nabla u_{\varphi + sw}(x^\prime,e^{\varphi + sw}) \cdot \nu_{\varphi + sw} \right) \nonumber \\
& = \left. \frac{d}{ds}\right|_{s = 0} (\nabla u_{\varphi + sw}(x', e^{\varphi + sw})) \cdot \nu_\varphi \nonumber \\
& \ \ \ + \nabla u_\varphi (x^\prime, e^\varphi)\cdot \left. \frac{d}{ds}\right|_{s = 0}  \nu_{\varphi + sw} \label{eq:d_normal_dev_ds}
\end{align}
where $\nu_\varphi$ is given by \eqref{eq:nu} and
\begin{equation*}
\nu_{\varphi + sw} = \frac{(- e^{\varphi + sw} \nabla_{\mathbb R^{N - 1}} (\varphi + sw), 1)}{\sqrt{1 + |e^{\varphi + sw} \nabla_{\mathbb R^{N - 1}}(\varphi + sw)|^2}}.
\end{equation*}
Now, for the first term in the right-hand side of \eqref{eq:d_normal_dev_ds}, thanks to the argument presented in \cite[Lemma 3.2]{IacopettiPacellaWeth2022}, we have
\begin{equation*}
\frac{d}{ds} (\nabla u_{\varphi + sw}) = \nabla \left(\frac{d}{ds} u_{\varphi + sw} \right),
\end{equation*}
and thus we obtain
\begin{equation}
\label{eq:d_nabla_u_phi_sw_ds}
\left.\frac{d}{ds}\right|_{s = 0}  (\nabla u_{\varphi + sw} (x', e^{\varphi + sw})) = \nabla \widetilde u_w(x^\prime,e^\varphi) + D^2u_\varphi(x^\prime,e^\varphi) (0^\prime,we^\varphi).
\end{equation}
On the other hand, for the last term in \eqref{eq:d_normal_dev_ds}, we check that
\begin{align}
\left. \frac{d}{ds} \right|_{s = 0} \nu_{\varphi + sw}
& = -  \frac{e^\varphi}{\sqrt{1 + |e^\varphi \nabla_{\mathbb R^{N - 1}} \varphi|^2}} (\nabla_{\mathbb R^{N - 1}} w + w \nabla_{\mathbb R^{N - 1}} \varphi, 0) \nonumber \\
& \quad -  \frac{(e^\varphi)^2 (w |\nabla_{\mathbb R^{N - 1}} \varphi|^2 + \nabla_{\mathbb R^{N - 1}} \varphi \cdot \nabla_{\mathbb R^{N - 1}} w)}{1 + |e^\varphi \nabla_{\mathbb R^{N - 1}} \varphi|^2} \nu_{\varphi} \label{eq:d_nu_s_ds}
\end{align}
Finally, substituting \eqref{eq:d_normal_dev_ds}--\eqref{eq:d_nu_s_ds} into \eqref{eq:second_variation} we obtain \eqref{eq:T2_cartesian}.
\end{proof}

As in Section \ref{sec:cone}, in view of Definition \ref{def:energy_stationary}, we consider a volume constraint. In the case of cartesian graphs, the volume of the domain $\Omega_\varphi$ associated to $\varphi \in C^2(\overline \omega)$ is expressed by
\begin{equation}
\label{eq:volume_cartesian}
\mathcal V(\varphi) = |\Omega_\varphi| = \int_\omega e^\varphi \ dx'.
\end{equation}
The functional $\mathcal V$ is of class $C^2$ and for every $v, w \in C^2(\overline \omega)$ it holds
\begin{equation}
\label{eq:derivatives_V_cartesian}
\mathcal V'(\varphi) [v] = \int_\omega e^\varphi v \ dx', \qquad \mathcal V''(\varphi)[v, w] = \int_\omega e^\varphi v w \ dx'.
\end{equation}

For $c > 0$ we define the manifold
\begin{equation*}
M \coloneqq \left\{\varphi \in C^2(\overline \omega) \ : \ \int_\omega e^\varphi \ dx' = c\right\},
\end{equation*}
whose tangent space at any point $\varphi \in M$ is given by
\begin{equation}
\label{eq:tangent_space_cylinder}
T_\varphi M = \left\{v \in C^2(\overline \omega) \ : \ \int_\omega e^\varphi v \ dx' = 0 \right\}.
\end{equation}

We consider the restricted functional
\begin{equation*}
I(\varphi) = T|_M(\varphi), \quad \varphi \in M.
\end{equation*}

As before, if $\varphi \in M$ is a critical point for $I$, then there exists a Lagrange multiplier $\mu \in \mathbb R$ such that
\begin{equation*}
T'(\varphi) = \mu I'(\varphi).
\end{equation*}
Results analogous to Proposition \ref{prop:Lagrange_multiplier_is_negative} and Lemma \ref{lem:Lagrange_multipliers_second_derivatives} hold with the same proofs. In particular, we point out that for an energy stationary pair $(\Omega_\varphi, u_\varphi)$ under a volume constraint the function $u_\varphi$ has constant normal derivative on $\Gamma_\varphi$. For the reader's convenience, we restate here these results.

\begin{prop}
\label{prop:Lagrange_multiplier_is_negative_cylinder}
Let $\varphi \in M$ and let $(\Omega_\varphi, u_\varphi)$ be energy-stationary under a volume constraint. Then the Lagrange multiplier $\mu$ is negative and
\begin{equation*}
\frac{\partial u_\varphi}{\partial \nu} = - \sqrt{- 2 \mu} \quad \text{ on } \quad \Gamma_\varphi.
\end{equation*} 
\end{prop}

\begin{proof}
The same as in \cite[Lemma 4.1]{IacopettiPacellaWeth2022}
\end{proof}

For the second derivative of $I$ we have
\begin{lemma}
\label{lem:Lagrange_multipliers_second_derivatives_cylinder}
Let $\varphi \in M$ and let $v, w \in T_\varphi M$. If $(\Omega_\varphi, u_\varphi)$ is energy-stationary under a volume constraint, then 
\begin{equation}
\label{eq:Lagrange_multipliers_second_derivatives_cylinder}
I''(\varphi)[v, w] = T''(\varphi)[v, w] - \mu \mathcal V''(\varphi)[v, w].
\end{equation}
\end{lemma}

\begin{proof}
The same as in \cite[Lemma 4.3]{IacopettiPacellaWeth2022}
\end{proof}

\subsection{The case $\mathbf{\varphi \equiv 0}$ and one-dimensional solutions}

When $\varphi \equiv 0$ (that is, $\Gamma_\varphi = \Gamma_0$ is the intersection of the cylinder with the plane $x_N = 1$), the domain $\Omega_0$ is just the finite cylinder $$\Omega_\omega:=\omega \times (0, 1).$$ 
Then, if $f$ is a locally Lipschitz continuous function, any weak solution of \eqref{eq:pde_cylinder} is also a classical solution up to the boundary, i.e., it belongs to $C^2(\overline \Omega_\omega)$. This follows  by standard regularity theory by considering the boundary conditions and that $\partial \Omega_\omega$ is made by the union of three $(N - 1)$-dimensional manifolds (with boundary) intersecting orthogonally (see also \cite[Proposition 6.1]{PacellaTralli2020}).

In $\Omega_\omega$, for suitable nonlinearities, we can find a solution of \eqref{eq:pde_cylinder} in $\Omega_\omega$ which depends only on $x_N$ in the following way: first, we can apply some variational method to find a solution $u$ of the ordinary differential equation
\begin{equation}
\label{eq:nonlinear_ode_cylinder}
\begin{cases}
-u'' = f(u) \quad \text{ in } (0, 1) \\
u'(0) = u(1) = 0
\end{cases}
\end{equation}
and then set 
$$u_\omega(x', x_N) := u(x_N), \quad \text{$(x', x_N) \in \Omega_\omega$}.$$ 
Recall that, in one dimension, there is no critical Sobolev exponent for the  embedding into $L^p$. So one example of a suitable nonlinearity is $f(u) = u^p$ with $1 < p < \infty$, or those of Proposition \ref{prop:existence_radial_sector} with the only caution that in (iii), for $N\geq 2$ we can take $1 < p < \infty$. 

For our purposes we need to consider one-dimensional solutions $u_\omega$ of \eqref{eq:pde_cylinder} in $\Omega_\omega$ that are nondegenerate, which means that the linearized operator
\begin{equation*}
L_{u_\omega} = - \Delta - f'(u_\omega)
\end{equation*}
does not admit zero as an eigenvalue. In other words, $u_\omega$ is nondegenerate if there are no nontrivial weak solutions $\phi \in H_0^1(\Omega_\omega \cup \Gamma_{1, 0})$ of the problem
\begin{equation}
\label{eq:linearized_problem_cylinder}
\left\{
\begin{array}{rcll}
- \Delta \phi - f'(u_\omega) \phi & = & 0 & \quad \text{ in } \Omega_\omega \\[5pt]
\phi & = & 0 & \quad \text{ on } \Gamma_0 \\[3pt]
\displaystyle \frac{\partial \phi}{\partial \nu} & = & 0 & \quad \text{ on } \Gamma_{1, 0}
\end{array}
\right.
\end{equation}

To analyze the spectrum of $L_{u_\omega}$ it is convenient to consider the following auxiliary one-dimensional eigenvalue problem:
\begin{equation}
\label{eq:linearized_eigenvalue_problem_cylinder}
\begin{cases}
- z'' - f'(u_\omega) z = \alpha z \quad \text{ in } (0, 1) \\
z'(0) = z(1) = 0
\end{cases}
\end{equation}

We denote the eigenvalues of \eqref{eq:linearized_eigenvalue_problem_cylinder} by $\alpha_i$, for $i\in \mathbb{N}$. Clearly, they correspond to the eigenvalues of the linear operator
\begin{equation}
\widehat L_{u_\omega}(z) = - z'' - f'(u_\omega) z
\end{equation}
with the boundary conditions of \eqref{eq:linearized_eigenvalue_problem_cylinder}.

We also consider the following Neumann eigenvalue problem in the domain $\omega \subset \mathbb R^{N - 1}$:
\begin{equation}
\label{eq:Neumann_problem_on_omega}
\left\{
\begin{array}{rcll}
- \Delta_{\R^{N-1}} \psi & = & \lambda \psi & \quad \text{ in } \omega \\[4pt]
\displaystyle \frac{\partial \psi}{\partial \nu_{\partial \omega}} & = & 0 & \quad \text{ on } \partial \omega
\end{array}
\right.
\end{equation}
where $- \Delta_{\R^{N-1}} = - \sum_{i = 1}^{N - 1} \frac{\partial^2}{\partial x_i^2}$ is the Laplacian in $\mathbb R^{N - 1}$, i.e. with respect to the variables $x_1,\ldots,x_{N-1}$. We denote its eigenvalues by
\begin{equation}
0 = \lambda_0(\omega) < \lambda_1(\omega) \leq \lambda_2 (\omega) \leq  \ldots.
\end{equation}
It is well-known that $\lambda_j(\omega) \nearrow + \infty$ as $j \to \infty$ and that the normalized eigenfunctions form a basis $(\psi_j)_j$ of the tangent space $T_0M$ defined in \eqref{eq:tangent_space_cylinder} when $\varphi \equiv 0$.

\begin{lemma}
\label{lem:sum_of_spectra_cylinder}
The spectra of $L_{u_\omega}$, $\widehat L_{u_\omega}$ and $- \Delta_{\R^{N-1}}$ with respect to the above boundary conditions are related by
\begin{equation}
\sigma(L_{u_\omega}) = \sigma(\widehat L_{u_\omega}) + \sigma(- \Delta_{\R^{N-1}}).
\end{equation}
\end{lemma}

\begin{proof}
We begin by showing that $\sigma(L_{u_\omega}) \subset \sigma(\widehat L_{u_\omega}) + \sigma(- \Delta_{\R^{N-1}})$. Let $\tau \in \sigma(L_{u_\omega})$ and let $\phi \in H_0^1(\Omega_\omega \cup \Gamma_{1,0})$ be an associated eigenfunction, that is, $\phi$ is a weak solution of
\begin{equation}
\label{eq:linearized_problem_cylinder_with_eigenvalue}
\left\{
\begin{array}{rcll}
- \Delta \phi - f'(u_\omega) \phi & = & \tau \phi & \quad \text{ in } \Omega_\omega \\[5pt]
\phi & = & 0 & \quad \text{ on } \Gamma_0 \\[3pt]
\displaystyle \frac{\partial \phi}{\partial \nu} & = & 0 & \quad \text{ on } \Gamma_{1, 0}
\end{array}
\right.
\end{equation}
As observed at the beginning of this subsection for the the nonlinear problem \eqref{eq:pde_cylinder}, by the shape of $\Omega_\omega$ and the boundary conditions, since $f \in C^{1, \alpha}(\mathbb R)$, by standard elliptic regularity, we have that $\phi$ is a classical solution of \eqref{eq:linearized_problem_cylinder_with_eigenvalue} in $\overline \Omega_\omega$.

Let $\lambda$ be an eigenvalue of $- \Delta_{\R^{N-1}}$ with homogeneous Neumann boundary condition on $\omega$ and let $\psi$ be an associated eigenfunction. Define
\begin{equation}
z(x_N) := \int_\omega \phi(x', x_N) \psi(x') \ dx'.
\end{equation}

Then, differentiating with respect to $x_N$, using Green's formulas and the boundary conditions we have
\begin{align}
- z'' 
& = \int_\omega - \frac{\partial^2 \phi}{\partial x_N^2} \psi \ dx' \nonumber \\
& = \int_\omega (- \Delta \phi + \Delta_{\R^{N-1}} \phi) \psi \ dx' \nonumber \\
& = \int_\omega f'(u_\omega) \phi \psi \ dx' + \int_\omega \tau \phi \psi \ dx' + \int_\omega \Delta_{\R^{N-1}} \psi \phi \ dx' \nonumber \\
& = f'(u_\omega) z + \tau z - \lambda z.\nonumber
\end{align}
Thus $(\tau - \lambda) \in \sigma(\widehat L_{u_\omega})$ and hence $\tau = (\tau - \lambda) + \lambda \in \sigma(\widehat L_{u_\omega}) + \sigma(- \Delta_{\R^{N-1}})$.

To show the reverse inclusion, let $\alpha \in \sigma(\widehat L_{u_\omega})$, $\lambda \in \sigma( - \Delta_{\R^{N-1}})$ and let $z, \psi$ be, respectively, the associated eigenfunctions. Setting for  $x = (x', x_N)\in\Omega_\omega$
\begin{equation*}
\phi(x', x_N) := z(x_N) \psi(x'),
\end{equation*}
we note that  
\begin{align}
- \Delta \phi 
& = - z'' \psi - \Delta_{\R^{N-1}} \psi  z \nonumber \\
& = f'(u_\omega) z \psi + \alpha z \psi + \lambda z \psi\\\nonumber 
& = f'(u_\omega) \phi + (\alpha  + \lambda)\phi.
\end{align}
Finally, by construction, we easily check that $\phi$ satisfies the boundary conditions of \eqref{eq:linearized_problem_cylinder_with_eigenvalue}.
As a consequence, we deduce that
\begin{equation*}
\alpha + \lambda \in \sigma(L_{u_\omega})
\end{equation*}
and this concludes the proof.
\end{proof}

\begin{cor}
The problem \eqref{eq:linearized_problem_cylinder} admits zero as an eigenvalue if and only if there exist $i \in \mathbb N^+$ and $j \in \mathbb N$ such that 
\begin{equation*}
\alpha_i + \lambda_j(\omega) = 0
\end{equation*}
holds.
\end{cor}
\begin{proof}
It follows immediately from Lemma \ref{lem:sum_of_spectra_cylinder}.
\end{proof}

\begin{cor}
A one-dimensional solution of \eqref{eq:pde_cylinder} is nondegenerate if both the following conditions are satisfied:
\begin{enumerate}
\item the eigenvalue problem \eqref{eq:linearized_eigenvalue_problem_cylinder} in $(0, 1)$ does not admit zero as an eigenvalue;\\[-3mm]
\item $\lambda_1(\omega) > - \alpha_1$.
\end{enumerate}
\end{cor}

\begin{proof}
Analogous to the proof of Corollary \ref{cor:nondegeneracy_cone}.
\end{proof}

\subsection{Stability/instability of the pair $\mathbf{(\Omega_\omega, \textit{u}_\omega)}$}

In this subsection, we prove a general stability/instability theorem for the pair $(\Omega_\omega, u_\omega)$. We begin with some preliminary results.

Firstly, we recall that when $\varphi \equiv 0$ the tangent space $T_0M$ is given by
\begin{equation}
T_0M = \left\{v \in C^2(\overline \omega) \ : \ \int_\omega v \ dx' = 0 \right\}.
\end{equation}
Since $u_\omega$ depends on $x_N$ only, in order to simplify the notations, we denote with a prime the derivative with respect to $x_N$, and thus we write
\begin{equation*}
u_\omega^\prime(x_N)=u_\omega^\prime(x^\prime,x_N) \coloneqq \frac{\partial u_\omega}{\partial x_N}(x', x_N).
\end{equation*}
Then, for $v \in T_0 M$, we have that the function $\widetilde u$ (see \eqref{eq:equation_for_u_tilde_cartesian_graphs}), which belongs to $H^1(\Omega_\omega)$, is a weak solution of   
\begin{equation}
\label{eq:equation_for_u_tilde_phi_zero}
\left\{
\begin{array}{rcll}
- \Delta \widetilde u & = & f'(u_\omega) \widetilde u & \quad \text{ in } \Omega_\omega \\ [5pt]
\widetilde u & = & - u_\omega'(1) v & \quad \text{ on } \Gamma_0 \\ [3pt]
\displaystyle \frac{\partial \widetilde u}{\partial \nu} & = & 0 & \quad \text{ on } \Gamma_{1, 0}
\end{array}
\right.
\end{equation}

As before, by elliptic regularity we know that $\widetilde u$ is regular in $\overline \Omega_\omega$, and thus it is a classical solution. We also note that, by the nondegeneracy of $u_\omega$, there exists a unique solution of \eqref{eq:equation_for_u_tilde_phi_zero}.

\begin{lemma}
\label{thm:exists_h}
Let $\lambda_j > 0$ be any positive eigenvalue for the Neumann problem \eqref{eq:Neumann_problem_on_omega} and let $\psi_j$ be any normalized eigenfunction associated to $\lambda_j$. Let $\widetilde u_j \in H^1(\Omega_\omega)$ be the solution of \eqref{eq:equation_for_u_tilde_phi_zero} with $v = \psi_j$. Then the function
\begin{equation}
\label{eq:def_of_h}
h_j(x_N) := \int_\omega \widetilde u_j(x', x_N) \psi_j(x') \ dx', \quad x_N \in (0, 1]
\end{equation}
satisfies
\begin{equation}
\label{eq:equation_for_h_cylinder}
\begin{cases}
- h_j'' - f'(u_\omega)h_j = - \lambda_j h_j \quad \text{ in } \quad (0, 1) \\[2pt]
h_j(1) = - u_\omega'(1) \\[2pt]
h_j'(0) = 0 
\end{cases}
\end{equation}
\end{lemma}

\begin{proof}
For simplicity of notation we drop the index $j$ and simply write $\tilde u$, $h$, $\psi$ and $\lambda$ instead of $\tilde u_j$, $h_j, \psi_j$ and $\lambda_j$.

First observe that, as $\tilde u=- u_\omega'(1)\psi$ on $\Gamma_0$, we have
\begin{equation*}
h(1) = \int_\omega - u_\omega'(1) \psi^2 \ dx' = - u_\omega'(1).
\end{equation*}

Now, differentiating with respect to $x_N$ under the integral sign and using Green's formula, taking into account the boundary conditions, we have
\begin{align}
- h'' 
& = \int_\omega - \frac{\partial^2 \widetilde u}{\partial x_N^2} \psi \ dx' = \int_\omega (- \Delta \widetilde u + \Delta_{\R^{N-1}}\widetilde u) \psi \ dx' \nonumber \\
& = \int_\omega f'(u_\omega) \widetilde u \psi \ dx' + \int_\omega \Delta_{\R^{N-1}} \widetilde u \psi \ dx'\nonumber \\
& = f'(u_\omega) h + \int_\omega \widetilde u \Delta_{\R^{N-1}} \psi \ dx' \nonumber \\
& = f'(u_\omega) h - \lambda \int_\omega \widetilde u \psi \ dx'  = f'(u_\omega)h - \lambda h. \nonumber
\end{align}

Finally, exploiting the Neumann condition for $\widetilde u$ on $\Gamma_{1, 0}$, we check that $h'(0) = 0$.
\end{proof}
\begin{remark}
\label{rem:u_tilde_cylinder} Note that for $\widetilde u_j$, $h_j$ as in Lemma \ref{thm:exists_h} we have that
\begin{equation*}
\widetilde u_j (x', x_N) = h_j(x_N) \psi_j(x').
\end{equation*}
Indeed:
\begin{align}
- \Delta (h_j(x_N) \psi_j(x')) 
& = - h_j(x_N) \Delta_{\R^{N-1}} \psi_j(x') - h_j''(x_N) \psi_j(x') \nonumber \\
& = \lambda_j h_j(x_N) \psi_j(x') + f'(u_\omega)h_j(x_N) \psi_j(x') - \lambda_j h_j(x_N) \psi_j(x') \nonumber \\
& = f'(u_\omega) \widetilde u_j.\nonumber
\end{align}
\end{remark}

Moreover, by \eqref{eq:equation_for_h_cylinder} and \eqref{eq:Neumann_problem_on_omega}, the function $h_j \psi_j$ satisfies the boundary conditions in \eqref{eq:equation_for_u_tilde_phi_zero}, so that $h_j \psi_j$ is the unique solution of \eqref{eq:equation_for_u_tilde_phi_zero} and thus coincides with $\widetilde u_j$.

\begin{prop}
\label{prop:h_positive_cylinder}
Let $j \geq 1$, $\lambda_j$ be a positive Neumann eigenvalue of $- \Delta_{\R^{N-1}}$ in $\omega$, and let $h_j$ be the solution of \eqref{eq:equation_for_h_cylinder}. Assume that $-\alpha_1 < \lambda_j$, where $\alpha_1$ is the smallest eigenvalue of \eqref{eq:linearized_eigenvalue_problem_cylinder}.   Then it holds that
\begin{equation*}
h_j > 0 \quad \text{ in } [0, 1].
\end{equation*}
\end{prop}

\begin{proof}
We can reflect $h_j$ by evenness with respect to $0$ to have a solution of the linear problem
\begin{equation}\label{eq:ODEhj}
\begin{cases}
- h_j'' - f'(u_\omega) h_j + \lambda_j h_j = 0 \quad \text{ in } (-1, 1) \\[2pt]
h_j(-1) = h_j(1) = -u_\omega'(1) > 0.
\end{cases}
\end{equation}

By reflection and \eqref{eq:linearized_eigenvalue_problem_cylinder}, the first eigenvalue of the linear operator $$z'' - f'(u_\omega)z \quad \text{in $(0, 1)$}$$  with the boundary condition $z(-1) = z(1) = 0$ is exactly $\alpha_1$. Therefore the first eigenvalue of the linear operator $$\widetilde L_{u_\omega}{g} = - g''- f'(u_\omega)g + \lambda_j g$$ with zero boundary condition in $(-1, 1)$ is $\beta_1 = \alpha_1 + \lambda_j$.

It is well-known that $\widetilde L_{u_\omega}$ satisfies the maximum principle whenever $\beta_1 > 0$, i.e., when $\lambda_j > - \alpha_1$. Therefore, by \eqref{eq:ODEhj}, the function $h_j$ satisfies $h_j \geq 0$ in $(-1,1)$, and by the strong maximum principle we conclude that $h_j > 0$ in $(-1, 1)$.
\end{proof}

We can now state and prove the main result of this section.

\begin{theorem}\label{thm:generalinststab}
Let $\omega \subset \mathbb R^{N - 1}$ be a smooth bounded domain. Let $f\in C^{1, \alpha}_{loc}(\R)$ such that there exists a positive one-dimensional non-degenerate solution $u_\omega$ of \eqref{eq:pde} in $\Omega_\omega$, and let $h_1$ be the solution to \eqref{eq:equation_for_h_cylinder} with $j=1$. Let $\lambda_1=\lambda_1(\omega)$ be the first non-trivial eigenvalue of $-\Delta_{\R^{N-1}}$ with homogeneous Neumann conditions, let $\alpha_1$ be the first-eigenvalue of \eqref{eq:nonsingular_eigenvalue_problem_intro} and let $\rho$ be the number defined by
\begin{equation}\label{eq:teoassumpt2}
\rho \coloneqq -  f(u_\omega(0))h_1(0) -\lambda_1\int_0^1 h_1u_\omega^\prime\, dx_N .
\end{equation}
Assume that $\lambda_1 > - \alpha_1$. Then
\begin{itemize}
\item[(i)]  if $\rho<0$, then $(\Omega_\omega, u_\omega)$ is an unstable energy-stationary pair;
\item[(ii)] if $\rho>0$, then $(\Omega_\omega, u_\omega)$ is  a stable energy stationary pair.
\end{itemize}
\end{theorem}

\begin{proof}
We first observe that since $\frac{\partial u_\omega}{\partial \nu}$ is constant on $\Gamma_0$ then, by the analogous of Proposition \ref{prop:Lagrange_multiplier_is_negative} for cylinders, we infer that the pair $(\Omega_\omega, u_\omega)$ is an energy-stationary pair. 

Let $w \in T_0M$ and assume without loss of generality that $\int_\omega w^2 \ dx' = 1$. In order to prove $(i)$-$(ii)$ we first determine a suitable expression for $I''(0)[w, w]$. To this end, for each $j \in \mathbb N^+$, let $\widetilde u_j$ be the solution of \eqref{eq:equation_for_u_tilde_phi_zero} with $v = \psi_j$ and let $h_j$ be the solution of \eqref{eq:equation_for_h_cylinder}. Then we can write
\begin{equation*}
w = \sum_{j = 1}^\infty (w, \psi_j) \psi_j
\end{equation*}
where $(\cdot, \cdot)$ is the inner product in $L^2(\omega)$. Moreover, we can check that
\begin{equation*}
\widetilde u = \sum_{j = 1}^\infty (w, \psi_j) \widetilde u_j
\end{equation*}
is the solution of \eqref{eq:equation_for_u_tilde_phi_zero} corresponding to $w$. Then, taking $\varphi=0$ in Lemma \ref{lem:T2_cartesian}, exploiting Lemma \ref{lem:Lagrange_multipliers_second_derivatives_cylinder}, taking into account that by Proposition \ref{prop:Lagrange_multiplier_is_negative_cylinder} the Lagrange multiplier $\mu$ is given by
\begin{equation*}
\mu = - \frac{1}{2}(u_\omega'(1))^2,
\end{equation*}
by Remark \ref{rem:u_tilde_cylinder} and observing that $\nabla u_\omega \perp (\nabla_{\mathbb R^{N - 1}} w, 0)$, we infer that

\begin{align}
I''(0)[w, w]
& = - \frac{1}{2} \int_\omega (u_\omega'(1))^2 w^2 \ dx' \nonumber \\
& \quad -  \int_\omega u_\omega'(1) \left(\sum_{j = 1}^\infty (w, \psi_j) h_j'(1) \psi_j \right) \left(\sum_{k = 1}^\infty (w, \psi_k) \psi_k\right) \ dx' \nonumber \\
& \quad - \int_\omega u_\omega'(1) u_\omega''(1) w^2 \ dx' + \frac{1}{2} (u_\omega'(1))^2 \int_\omega w^2 \ dx' \nonumber \\
& = - u_\omega'(1) \int_\omega \left(\sum_{j = 1}^\infty (w, \psi_j)^2 h_j'(1) \psi_j^2 \ \right) dx' - u_\omega'(1) u_\omega''(1) \nonumber
\end{align}
Finally, since $u_\omega$ is a solution to \eqref{eq:nonlinear_ode_cylinder} we deduce that
\begin{equation}\label{eq:eq:Isecondzerobis}
I''(0)[w, w]=- u_\omega'(1) \int_\omega \left(\sum_{j = 1}^\infty (w, \psi_j)^2 h_j'(1) \psi_j^2 \ \right) dx' + u_\omega'(1) f(0).
\end{equation}
In particular, choosing $w=\psi_1$ and plugging it into \eqref{eq:eq:Isecondzerobis} we infer that
\begin{equation}\label{eq:expressionIsecondpsi1}
I''(0)[\psi_1, \psi_1]=- u_\omega'(1) h_1'(1) + u_\omega'(1) f(0).
\end{equation}
Multiplying the equation in \eqref{eq:equation_for_h_cylinder} (with $j=1$) by $u^\prime_\omega$ and integrating by parts we get
$$-(h_1^\prime u_\omega^\prime)\big|_0^1 + \int_0^1h_1^\prime u_\omega^{\prime\prime}\, dx_N = \int_0^1(f^\prime(u_\omega)-\lambda_1)h_1u_\omega^\prime\, dx_N.$$
Exploiting \eqref{eq:nonlinear_ode_cylinder}, integrating by parts and taking into account that $h_1(1)=-u_\omega^\prime(1)$ we obtain
\begin{equation}\label{eq:contointparts}
\begin{array}{lll}
&&\displaystyle -h_1^\prime(1)u_\omega^\prime(1) - \int_0^1h_1^\prime f(u_\omega)\, dx_N\\ &=&\displaystyle \int_0^1f^\prime(u_\omega)u_\omega^\prime h_1\, dx_N -\lambda_1\int_0^1  h_1 u_\omega^\prime\, dx_N\\
&=&\displaystyle (f(u_\omega)h_1)\big|_0^1 - \int_0^1 f(u_\omega) h_1'\, dx_N -\lambda_1\int_0^1  h_1 u_\omega^\prime\, dx_N\\
&=&\displaystyle -f(0)u_\omega^\prime(1)-f(u_\omega(0))h_1(0) - \int_0^1 f(u_\omega) h_1'\, dx_N -\lambda_1\int_0^1  h_1 u_\omega^\prime\, dx_N\\
\end{array}
\end{equation}
Hence, we deduce that
\begin{equation}\label{eq:contointparts2}
-h_1^\prime(1)u_\omega^\prime(1) =-f(0)u_\omega^\prime(1)-f(u_\omega(0))h_1(0)  -\lambda_1\int_0^1  h_1 u_\omega^\prime\, dx_N
\end{equation}
In the end, from \eqref{eq:expressionIsecondpsi1}, \eqref{eq:contointparts2} and recalling \eqref{eq:teoassumpt2}, we obtain
$$I''(0)[\psi_1, \psi_1]=-f(u_\omega(0))h_1(0)  -\lambda_1\int_0^1  h_1 u_\omega^\prime\, dx_N=\rho.
 $$
Therefore, if $\rho<0$ then $I''(0)[\psi_1, \psi_1]<0$, i.e., $(\Omega_\omega, u_\omega)$ is an unstable energy-stationary pair, and this proves (i).\\

Let us prove (ii). Let $w \in T_0M$ such that $\int_\omega w^2 \ dx' = 1$.  From \eqref{eq:eq:Isecondzerobis} we know that $I''(0)[w, w]=- u_\omega'(1) \int_\omega \left(\sum_{j = 1}^\infty (w, \psi_j)^2 h_j'(1) \psi_j^2 \ \right) dx' + u_\omega'(1) f(0)$. Thanks to the assumption $\lambda_1>-\alpha_1$ the following holds true.\\

\textbf{Claim:} if $k > j$, then 
\begin{equation}\label{eq:claimsect4}
h_k'(1) {\geq} h_j'(1),
\end{equation}
and actually $h_k'(1) > h_j'(1)$ if $\lambda_k > \lambda_j$.\\

Indeed, by definition $h_k$, $h_j$ satisfy, respectively, the following:
\begin{align}
& - h_k'' - f'(u_\omega) h_k = - \lambda_k h_k, \label{eq:forhk} \\
& - h_j'' - f'(u_\omega) h_j = - \lambda_j h_j. \label{eq:forhj}
\end{align}
Multiplying \eqref{eq:forhk} by $h_j$ and integrating on $(0, 1)$ we obtain
\begin{align}
\int_0^1 - h_k'' h_j \, d x_N 
& = \int_0^1 h_k' h_j' \, d x_N - (h_k' h_j)\big|_0^1 \nonumber \\
& = \int_0^1 f'(u_\omega) h_j h_k \, d x_N - \lambda_k \int_0^1 h_j h_k \, d x_N \label{eq:int_eq_for_h_k_times_h_j}
\end{align}
Similarly, multiplying \eqref{eq:forhj} by $h_k$, integrating on $(0, 1)$ and then subtracting the result from \eqref{eq:int_eq_for_h_k_times_h_j}, we obtain
\begin{equation}
\label{eq:h_k_prime_minus_h_j_prime}
 - (h_k' h_j - h_j' h_k)(1) = (\lambda_j - \lambda_k) \int_0^1 h_j h_k \, dx_N\ {\leq}\ 0, 
\end{equation}
because $h_j > 0$ and $h_k > 0$ (see Proposition \ref{prop:h_positive_cylinder}, which holds true for any $j\in\mathbb{N}^+$ because $\lambda_1>-\alpha_1$). Now, since $h_j(1)=h_k(1)=-u_\omega(1)$, then by \eqref{eq:h_k_prime_minus_h_j_prime} we deduce that
$$u_\omega'(1)(h_k'(1) - h_j'(1))\leq 0.$$
Hence, as $u_\omega'(1)<0$, Claim \eqref{eq:claimsect4} easily follows.\\

Now, thanks to \eqref{eq:eq:Isecondzerobis} and Claim \eqref{eq:claimsect4}, recalling again that $u'_\omega(1) < 0$ and exploiting \eqref{eq:contointparts2} it follows that
\begin{equation}\label{eq:teogeneral}
\begin{array}{lll}
\displaystyle I''(0)[w, w]&\geq&\displaystyle- u_\omega'(1)h_1^\prime(1)\int_\omega  \left(\sum_{j = 1}^\infty (w, \psi_j)^2\psi_j^2 \ \right) dx' + u_\omega'(1) f(0)\\[6mm]
&=&\displaystyle- u_\omega'(1)h_1^\prime(1) + u_\omega'(1) f(0)\\[2mm]
&=&\displaystyle -f(u_\omega(0))h_1(0)  -\lambda_1\int_0^1  h_1 u_\omega^\prime\, dx_N=\rho.
\end{array}
\end{equation}
Hence, if $\rho>0$ we have that $ I''(0)[w, w]>0$ for all $w \in T_0M$, i.e., $(\Omega_\omega, u_\omega)$ is a stable energy-stationary pair, and this proves (ii). The proof is complete.
\end{proof}

As a simple corollary of Theorem \ref{thm:generalinststab} we can now prove the stability/instability result of Theorem \ref{thm:sharp_stability_for_torsion}, which concerns the case of the torsional energy, i.e. when $f \equiv 1$.

\subsection*{Proof of Theorem \ref{thm:sharp_stability_for_torsion}}
When $f\equiv 1$ the eigenvalue problem \eqref{eq:linearized_eigenvalue_problem_cylinder} has only positive eigenvalues and therefore the condition $\lambda_1 > - \alpha_1$ is automatically satisfied. The only solution of 
\begin{equation}
\label{eq:pdecylinder}
\left\{
\begin{array}{rcll}
- \Delta u & = & 1 & \quad \text{ in } \Omega_\omega \\[4pt]
u & = & 0 & \quad \text{ on } \Gamma_0 \\[2pt]
\displaystyle \frac{\partial u}{\partial \nu} & = & 0 & \quad \text{ on } \Gamma_{1, 0}
\end{array}
\right.
\end{equation}
is the one-dimensional positive function given by 
\begin{equation}\label{eq:expressionuomegateo12}
u_\omega(x_N)=\frac{1-x_N^2}{2}.
\end{equation}
Clearly, as $u_\omega'(1)=-1$ and $f\equiv 1$, then for any $j\in\mathbb{N}^+$ \eqref{eq:equation_for_h_cylinder} reduces to
\begin{equation*}
\begin{cases}
- h_j''  + \lambda_j h_j = 0 \quad \text{ in } \quad (0, 1) \\[2pt]
h_j(1) = - u_\omega'(1) \\[2pt]
h_j'(0) = 0 
\end{cases}
\end{equation*}
whose unique solution is given by
$$h_j(x_N)=\frac{1}{\cosh(\sqrt{\lambda_j})}\cosh(\sqrt{\lambda_j} x_N).$$
In particular, taking $j=1$ and exploiting \eqref{eq:expressionuomegateo12} we can compute explicitly the number $\rho$ in \eqref{eq:teoassumpt2}, namely
$$\rho=-\frac{1}{\cosh(\sqrt{\lambda_1})} +\frac{\lambda_1}{\cosh(\sqrt{\lambda_1})}\int_0^1 \cosh(\sqrt{\lambda_1} x_N) x_N\, dx_N .
$$
Integrating by parts we readily check that
$$\int_0^1 \cosh(\sqrt{\lambda_1} x_N) x_N\, dx_N=\frac{\sinh(\sqrt{\lambda_1})}{\sqrt{\lambda_1}} -\frac{\cosh(\sqrt{\lambda_1})}{\lambda_1}+\frac{1}{\lambda_1},$$
and thus we obtain
\begin{equation}\label{eq:Psiteo12}
\rho = \sqrt{\lambda_1} \tanh(\sqrt{\lambda_1})-1.
\end{equation}
Let us consider the function $g:[0,+\infty[ \to \R$, defined by $g(t)= \sqrt{t} \tanh(\sqrt{t})-1$. Clearly $g(0)=-1$ and $g(t)\to +\infty$ as $t\to +\infty$ and by monotonicity we infer that $g$ has a unique zero in $]0,+\infty[$. We denote it by $\beta$ and from the previous argument and \eqref{eq:Psiteo12} we infer that $\rho<0$ if and only if $\lambda_1<\beta$. Then, by Theorem \ref{thm:generalinststab}-(i) we get that  $(\Omega_\omega, u_\omega)$ is an unstable energy-stationary pair, and this proves (i). 

Analogously, as $\rho>0$ if and only if $\lambda_1>\beta$, from Theorem \ref{thm:generalinststab}-(ii) we obtain that $(\Omega_\omega, u_\omega)$ is a stable energy-stationary pair. The proof is complete.
\qed \\

We conclude this section with the proof of Theorem \ref{thm:stability_semilinear_cylinder}.
\subsection*{Proof of Theorem \ref{thm:stability_semilinear_cylinder}} 
Let $w \in T_0M$ such that $\int_\omega w^2 \ dx' = 1$. Since $\lambda_1>-\alpha_1$, we can argue as in the proof of Theorem \ref{thm:generalinststab}-(ii), in particular, from the first two lines of \eqref{eq:teogeneral}, taking into account that, by assumption, $f(0)=0$, we have
\begin{equation}\label{eq:teo13Isecond}
I''(0)[w, w]\geq\displaystyle- u_\omega'(1)h_1^\prime(1).
\end{equation}
Now, since $h_1'' = (\lambda_1- f'(u_\omega)) h_1$ in $(0,1)$ and $h_1>0$ in $[0,1]$ by Proposition \ref{prop:h_positive_cylinder}, then, thanks to the assumption $\lambda_1>\sup_{x_N\in(0,1)}|f^\prime(u_\omega(x_N))|$ we infer that
$h_1''>0$ in $[0,1]$. In particular, as $h_1^\prime(0)=0$ we deduce that 
\begin{equation}\label{eq:teo13hprime}
h_1^\prime(1)>0.
\end{equation}
Finally, combining \eqref{eq:teo13Isecond} and \eqref{eq:teo13hprime} we obtain that $I''(0)[w, w]>0$ for all $w \in T_0M$, which means that $(\Omega_\omega, u_\omega)$ is a stable energy-stationary pair.
\qed

\begin{remark}
We notice that, if $f$ is a non-negative monotone increasing function, as in the case of the Lane-Emden nonlinearity \eqref{eq:LaneEmden}, then by the Gidas-Ni-Nirenberg theorem (\cite{GidasNiNirenberg1979}) and by the monotonicity of $f$ we infer that $\sup_{x_N\in(0,1)}|f^\prime(u_\omega(x_N))|=f^\prime(u_\omega(0))$. Thus the stability condition of Theorem \ref{thm:stability_semilinear_cylinder} reduces to
$$\lambda_1>f^\prime(u_\omega(0)).$$
\end{remark}

\begin{remark}
\label{rem:remark_on_numerics}
In the case of the Lane-Emden nonlinearity $f(u) = u^p$, at least for some integer values of $p$, it is possible to compute the solution $u_\omega$ numerically, as well as the eigenvalue $\alpha_1$ and the function $h_1$ for different values of $\lambda_1(\omega)$. This allows to compute $\rho$ numerically, so that, plotting the result for $\rho$ as a function of $\lambda_1(\omega)$, we obtain a region of instability for $\lambda_1(\omega)$ close to $- \alpha_1$.
\end{remark}

\section*{\small Acknowledgements}

\small We would like to thank David Ruiz for several useful discussions and Tobias Weth for pointing out a flaw in an early draft of the paper.

\bibliographystyle{acm}
\bibliography{ref_math}

\end{document}